\documentclass{amsart}

\usepackage{amsmath, amssymb, amsthm}
\usepackage{color}
\usepackage{url}
\usepackage{mathrsfs}
\usepackage{enumitem}
\usepackage{tikz}
\usepackage{float}

\calclayout

\newtheorem{theorem}{Theorem}[section]
\newtheorem{lemma}[theorem]{Lemma}
\newtheorem{proposition}[theorem]{Proposition}

\newtheorem{claim}[theorem]{Claim}

\newtheorem{fact}[theorem]{Fact}
\newtheorem{question}[theorem]{Question}
\newtheorem*{THM}{Main Theorem}
\newtheorem*{COR}{Main Corollary}
\newtheorem*{Claim 53}{Claim 5.3'}

\theoremstyle{definition}
\newtheorem{definition}[theorem]{Definition}
\newtheorem{remark}[theorem]{Remark}
\newtheorem{observation}[theorem]{Observation}
\newtheorem{example}[theorem]{Example}

\newcommand{\dom}{\mathrm{dom}}
\newcommand{\bb}{\mathbb}

\newcommand{\otp}{\mathrm{otp}}

\newcommand{\mb}{\mathbf}

\newcommand*\oline[1]{%
  \vbox{%
    \hrule height 0.5pt
    \kern0.25ex
    \hbox{%
      \kern -0.2em
      \ifmmode#1\else\ensuremath{#1}\fi
      \kern 0em
    }
  }
}

\title{Simultaneously vanishing higher derived limits without large cardinals}

\author{Jeffrey Bergfalk}
\address{Universit\"{a}t Wien \\
Institut f\"{u}r Mathematik \\
Kurt G\"{o}del Research Center \\
Kolingasse 14-16 \\
1010 Wien, Austria}
\email{jeffrey.bergfalk@univie.ac.at}
\thanks{The first author was partially supported by Austrian Science Foundation (FWF) Grant Number
Y1012-N35.}

\author{Michael Hru\v{s}\'{a}k}
\address{Centro de Ciencas Matem\'{a}ticas\\
UNAM\\
A.P. 61-3, Xangari, Morelia, Michoac\'{a}n\\
58089, M\'{e}xico}
\email{michael@matmor.unam.mx}
\thanks{The second author was partially supported by a CONACyT grant A1-S-16164 and PAPIIT grant IN104220.}

\author{Chris Lambie-Hanson}
\address{Department of Mathematics and Applied Mathematics \\
Virginia Commonwealth University \\
Richmond, VA 23284 \\ United States}
\email{cblambiehanso@vcu.edu}
\thanks{The work in this paper began during a visit by the third author to the
Centro de Ciencias Matem\'{a}ticas at UNAM Morelia while the first author was
a postdoc at that institution. Both authors would like to thank the CCM for its hospitality
and support.}

\begin{document}

\begin{abstract}
  A question dating to Sibe Marde\v{s}i\'{c} and Andrei Prasolov's 1988 work \cite{mp},
  and motivating a considerable amount of set theoretic work in the ensuing years, is that of
  whether it is consistent with the \textsf{ZFC} axioms for the higher derived limits
  $\mathrm{lim}^n$ $(n>0)$ of a certain inverse system $\mathbf{A}$ indexed by ${^\omega}\omega$ to
  simultaneously vanish. An equivalent formulation of this question is that of whether it is
  consistent for all $n$-coherent families of functions indexed by ${^\omega}\omega$ to be trivial.
  In this paper, we prove that, in any forcing extension given by adjoining $\beth_\omega$-many
  Cohen reals, $\mathrm{lim}^n \mathbf{A}$ vanishes for all $n > 0$. Our proof involves
  a detailed combinatorial analysis of the forcing extension and repeated applications
  of higher dimensional $\Delta$-system lemmas. This work removes all large cardinal
  hypotheses from the main result of \cite{SVHDL} and substantially reduces the
  least value of the continuum known to be compatible with the simultaneous vanishing
  of $\mathrm{lim}^n \mathbf{A}$ for all $n > 0$.
\end{abstract}

\subjclass[2010]{03E35, 03E75, 18E25, 55N07}

\keywords{Cohen forcing, derived limit, nontrivial coherence, Delta system lemma, strong homology}

\maketitle

\section{Introduction}
The set theoretic study of higher derived limits traces principally to Sibe Marde\v{s}i\'{c} and Andrei Prasolov's 1988 work \cite{mp}; it was in this paper that a relationship between \begin{enumerate}
\item the continuity properties of strong homology,
\item the behavior of the derived limits of inverse systems indexed by functions from $\omega$ to $\omega$, and
\item infinitary combinatorics and assumptions supplementary to \textsf{ZFC},
\end{enumerate}
was first perceived. The most elementary of the systems as in (2) was denoted $\mathbf{A}$ by Marde\v{s}i\'{c} and Prasolov in \cite{mp}, and the works which followed would show the behavior of its higher limits sensitive to a variety of set theoretic hypotheses; additional interest in these behaviors derived from their connection to the broader set theoretic theme of nontrivial coherence (main works in this line were \cite{dsv,todpp,kamo,todcmpct,far1,moorepfa,B1,bbm}; see \cite[Introduction]{SVHDL} for a brief research history). The outstanding question tracing to \cite{mp} was whether the statement ``$\mathrm{lim}^n\mathbf{A}=0$ for all $n>0$'' is consistent with the \textsf{ZFC} axioms; this was affirmatively  answered in \cite{SVHDL} in 2019, under the assumption of the existence of a weakly compact cardinal. Several immediately ensuing questions are listed in the conclusion of \cite{SVHDL}. The first of these, that of the consistency strength of this statement, is answered by our main result:
\begin{THM} The statement ``$\mathrm{lim}^n\mathbf{A}=0$ for all $n>0$'' holds in the extension of $V$ by the forcing $\mathrm{Add}(\omega,\beth_\omega)$ for adjoining $\beth_\omega$-many Cohen reals.
\end{THM}
In particular, the statement ``$\mathrm{lim}^n\mathbf{A}=0$ for all $n>0$'' carries no large cardinal strength whatsoever. The second of the questions listed in \cite{SVHDL} is that of the minimum value of the continuum compatible with this statement. The works \cite{mp} and \cite{SVHDL} established for that value lower and upper bounds of $\aleph_2$ and a weakly inaccessible cardinal, respectively; the gap between them was substantial.
\begin{COR} It is consistent relative to the \textsf{ZFC} axioms that $\mathrm{lim}^n\mathbf{A}=0$ for all $n>0$ and $2^{\aleph_0}=\aleph_{\omega+1}$.
\end{COR}
As noted in \cite{SVHDL}, plausible scenarios exist in which $\aleph_{\omega+1}$ is optimal, a point we return to in our conclusion below.

In order to describe the structure of our paper, we should first say a few words about our overall argument and, in particular, about how it both builds on and departs from that of \cite{SVHDL}. In both that work and this one, the idea is to argue in a given forcing extension that an arbitrary $n$-coherent family of functions $\Phi$ is trivial. In both cases, this is achieved in two steps: \begin{itemize}
\item First, a trivialization of the restricted family $\Phi\restriction A$ is found for some $A\subseteq{^\omega}\omega$.
\item Second, trivializations of $\Phi\restriction A$ are shown to extend to trivializations of all of $\Phi$.
\end{itemize}
The requirements of these two steps are in tension; what's wanted is an $A$ which is at once ``sufficiently small'' and ``sufficiently large'' to effect the first and second steps, respectively. This is a tension which the large cardinal assumption of \cite{SVHDL} may be viewed as resolving: there, the weak compactness of a cardinal $\kappa$ manifests as multidimensional $\Delta$-system relationships among large families of conditions in a finite support forcing iteration of length $\kappa$. These families' homogeneities lend them a ``smallness'' of the sort called for in step one; this being an iteration of Hechler forcings ensures that, nevertheless, the associated sets $A$ are $\leq^*$-cofinal in ${^\omega}\omega$, from which step two follows easily.

In the present work, cardinal arithmetic and inductive hypotheses on $n$ together take the place of the large cardinal assumption in \cite{SVHDL}. Here again, higher-dimensional $\Delta$-systems lie at the heart of step one, and we draw on \cite{hdds} for their description and analysis. Observe, however, that without large cardinal assumptions, such systems can only appear, in general, together with some drop in cardinality. In consequence, the set $A$ associated to such a system in step one of our argument is small in a much stronger sense than in \cite{SVHDL}. Nevertheless, in the context of Cohen forcing, genericity arguments coupled, at each stage $n$, with inductive hypotheses on the triviality of $k$-coherent families of functions for $k<n$ allow us to propagate the triviality of $\Phi\restriction A$ to all of $\Phi$ as desired.

Our account of this argument is structured as follows: in Section \ref{section2}, we record our basic conventions, some results on higher-dimensional $\Delta$-systems, and the conversion of assertions about $\mathrm{lim}^n\,\mathbf{A}$ to assertions about the triviality of $n$-coherent families of functions. This section includes a brief homological interlude that set theorists, for example,  may safely ignore: from Section \ref{highercoherencesection}, only Definition \ref{NCOH}, Fact \ref{mostmainfact}, and Proposition \ref{trivial_equivalence_fact} are needed  in the remainder of the paper. In Section \ref{n1section}, we prove a strong form of the $n=1$ instance of our main theorem. In Section \ref{propagatingsection}, we describe the ``step two'' portion of our argument outlined above, showing how, at each stage $n>1$, our trivializations of $\Phi\restriction A$ will extend to all of $\Phi$. In Section \ref{defining_trivializations_section}, we record a framework deriving from \cite{SVHDL} for defining trivializations of families like $\Phi\restriction A$. As the reader may have surmised, the paper from this point forward is fairly technical, and we therefore close this section with several heuristic comments. In Section \ref{highernsection}, we prove strong forms of the $n>1$ instances of our main theorem. In Section \ref{conclusion}, we conclude with a list of some of the more prominent questions arising in the field of set-theoretic research into higher derived limits, along with a discussion of some of their interrelations.
\section{Preliminaries}\label{section2}

\subsection{Notational conventions}\label{conventionssubsection}
If $X$ is a set and $\kappa$ is a cardinal, then $[X]^\kappa = \{Y \subseteq X
\mid |Y| = \kappa\}$ and $[X]^{<\kappa} = \{Y \subseteq X \mid |Y| < \kappa\}$.
If $\kappa$ and $\lambda$ are cardinals, then we say that $\lambda$ is
\emph{${<}\kappa$-inaccessible} if $\nu^{<\kappa} < \lambda$ for all $\nu < \lambda$.
If $X$ is a set of ordinals, then $\otp(X)$ denotes the order-type of $X$.
We will often view finite sets of ordinals as finite increasing sequences of
ordinals, and vice versa. For example, if $a \in [\mathrm{On}]^{<\omega}$ and
$\ell < \otp(a)$, then $a(\ell)$ is the unique $\alpha \in a$ such that
$|a \cap \alpha| = \ell$. If $\mb{m} \subseteq a$, then $a[\mb{m}] = \{a(\ell) \mid
\ell \in \mb{m}\}$. For any set $X$ of ordinals and natural number $n$, the notation $(\alpha_0, \ldots, \alpha_{n-1}) \in [X]^n$ will denote the
conjunction of the statements $\{\alpha_0, \ldots, \alpha_{n-1}\} \in [X]^n$ and
$\alpha_0 < \ldots < \alpha_{n-1}$. Frequently in what follows we index objects by finite sets, either of ordinals or of functions or of other finite sets. Our use of commas or curly brackets in the associated subscripts or superscripts is, in general, according to no other principle than readability. We will often, for example, write expressions like $q_{\alpha\beta}$ for expressions like $q_{\{\alpha, \beta\}}$; we handle these matters with some greater care, however, in the context of the more technical Section \ref{highernsection}.

The forcings appearing herein will all be of the form $\bb{P} = \mathrm{Add}(\omega, \chi)$,
where $\chi$ is an uncountable cardinal. We think of the conditions of
$\bb{P}$ as finite partial functions from $\chi \times \omega$ to $\omega$,
ordered by reverse inclusion. Forcing with $\bb{P}$ produces a generic function
$F:\chi \times \omega \rightarrow \omega$. For a fixed $\alpha < \chi$ we call the
function $F(\alpha, \cdot):\omega \rightarrow \omega$ the \emph{$\alpha^{\mathrm{th}}$
Cohen real added by $\bb{P}$}, and we will typically denote this function by
$f_\alpha$; we denote the canonical $\bb{P}$-name in $V$ for $f_\alpha$ by $\dot{f}_\alpha$. If $G$
is $\bb{P}$-generic over $V$ and $W \subseteq \chi$ then $G_W$ denotes
$\{p \in G \mid \dom(p) \subseteq W \times \omega\}$. For any condition $p$ in
$\bb{P}$ let $u(p)$ denote the set $\{\alpha < \chi \mid
\dom(p) \cap (\{\alpha\} \times \omega) \neq \emptyset\}$ and let
$\bar{p}$ denote the finite partial function from $\otp(u(p)) \times \omega$ to $\omega$
defined as follows: for all $i < \otp(u(p))$ and all $m < \omega$, define
$(i, j)$ to be in the domain of $\bar{p}$ if and only if $(u(p)(i), j) \in \dom(p)$; if so, let $\bar{p}(i,j) = p(u(p)(i),m)$. Intuitively,
$\bar{p}$ is a ``collapsed'' version of $p$. Notice that the set
$\{\bar{p} \mid p \in \bb{P}\}$ is a subset of the set of finite partial functions
from $\omega \times \omega$ to $\omega$ and is therefore countable.

For notational conventions pertaining more directly to coherent families of functions, see Section \ref{highercoherencesection} below.

\subsection{Higher-dimensional $\Delta$-systems}
Our proofs will make repeated use of multidimensional $\Delta$-system lemmas. For this purpose we recall some relevant definitions and results from \cite{hdds}.

\begin{definition} \label{aligned_def}
  Suppose that $a$ and $b$ are sets of ordinals.
  \begin{enumerate}
    \item We say that $a$ and $b$ are \emph{aligned} if $\otp(a) = \otp(b)$ and if $\otp(a \cap \gamma) = \otp(b \cap \gamma)$ for all
    $\gamma \in a \cap b$.
    In other words, if $\gamma$ is a common element of two aligned sets $a$ and $b$, then it
    occupies the same relative position in both $a$ and $b$.
    \item If $a$ and $b$ are aligned then we let $\mb{r}(a,b) :=
    \{i < \otp(a) \mid a(i) = b(i)\}$. Notice that, in this case,
    $a \cap b = a[\mb{r}(a,b)] = b[\mb{r}(a,b)]$.
  \end{enumerate}
\end{definition}

\begin{definition} \label{delta_system_def}
  Suppose that $H$ is a set of ordinals, $n$ is a positive integer, and $u_b$ is a set of ordinals
  for all $b \in [H]^n$.
  We call $\langle u_b \mid b \in [H]^n \rangle$ a \emph{uniform
  $n$-dimensional $\Delta$-system} if there is an ordinal $\rho$ and, for
  each $\mb{m} \subseteq n$, a set $\mb{r}_{\mb{m}} \subseteq \rho$
  satisfying the following statements.
  \begin{enumerate}
    \item $\otp(u_b) = \rho$ for all $b \in [H]^n$.
    \item For all $a,b \in [H]^n$ and $\mb{m} \subseteq n$, if $a$ and $b$ are aligned with $\mb{r}(a,b) = \mb{m}$,
    then $u_a$ and $u_b$ are aligned with $\mb{r}(u_a, u_b) = \mb{r}_{\mb{m}}$.
    \item For all $\mb{m}_0, \mb{m}_1 \subseteq n$, we have
    $\mb{r}_{\mb{m}_0 \cap \mb{m}_1} = \mb{r}_{\mb{m}_0} \cap \mb{r}_{\mb{m}_1}$.
  \end{enumerate}
\end{definition}

The following is a crucial feature of these $\Delta$-systems:

\begin{lemma} \label{extension_lemma}
  Suppose that $1 \leq n < \omega$ and $\langle u_b \mid b \in [H]^n \rangle$ is
  a uniform $n$-dimensional $\Delta$-system, as witnessed by $\rho$ and
  $\langle \mb{r}_{\mb{m}} \mid \mb{m} \subseteq n \rangle$, and assume for
  simplicity that $H$ has no largest element. For each $m < n$ and each $a \in [H]^{m}$, define a set
  $u_a$ by choosing $b \in [H]^n$ such that $b[m] = a$ and
  setting $u_a = u_b[\mb{r}_m]$. (Here and in similar places later in the paper,
  $m$ denotes the set of natural numbers less than $m$, so, for instance,
  $b[m] = \{b(\ell) \mid \ell < m\}$, $\mb{r}_m = \mb{r}_{\{\ell \mid \ell < m\}}$, and the statement that $b[m] = a$ amounts to
  asserting that $b$ end-extends $a$.) Then the following hold.
  \begin{enumerate}
    \item These definitions are independent of our choice of $b$.
    \item For each $a \in [H]^{<n}$, the collection
    \[
      \{u_{a \cup \{ \beta \}} \mid \beta \in H \setminus (\max(a) + 1)\}
    \]
    is a (1-dimensional) $\Delta$-system with root $u_a$.
  \end{enumerate}
\end{lemma}

\begin{proof}
  We first show (1). Indeed, fix $a \in [H]^{<n}$ and suppose that $b, b' \in [H]^n$ are such that
  $b[m] = a = b'[m]$. Since $H$ has no largest element, we can find
  $b'' \in [H]^n$ such that $b''[m] = a$ and $b''(m) > \max(b \cup b')$.
  In particular, $b$ and $b''$ are aligned and, likewise, $b'$ and $b''$ are aligned.
  Moreover, we have $m \subseteq \mb{r}(b, b'')$ and $m \subseteq \mb{r}(b', b'')$.
  It follows that $u_b$ and $u_{b''}$ are aligned and $u_b[\mb{r}_m] =
  u_{b''}[\mb{r}_m]$. Similarly, $u_{b'}[\mb{r}_m] = u_{b''}[\mb{r}_m]$, so
  $u_b[\mb{r}_m] = u_{b'}[\mb{r}_m]$.

  We now show (2). Fix $a \in [H]^{<n}$, let $m = |a|$, and suppose that
  $\beta < \beta'$ are elements of $H \setminus (\max(a) + 1)$. Fix
  ordinals $\gamma_{m + 1} < \gamma_{m + 2} < \ldots < \gamma_{n-1}$
  in $H \setminus (\beta' + 1)$. Let $b = a \cup \{\beta\} \cup \{\gamma_{m + 1},
  \ldots, \gamma_{n-1}\}$ and $b' = a \cup \{\beta'\} \cup \{\gamma_{m + 1}, \ldots,
  \gamma_{n-1}\}$. Then $b$ and $b'$ are aligned, with $\mb{r}(b, b') =
  n \setminus \{m\}$, and hence $u_b \cap u_{b'} = u_b[\mb{r}_{n \setminus \{m\}}] =
  u_{b'}[\mb{r}_{n \setminus \{m\}}]$. Moreover, we have $u_{a \cup \{\beta\}} = u_b[\mb{r}_{m + 1}]$
  and $u_{a \cup \{\beta'\}} = u_{b'}[\mb{r}_{m+1}]$. Putting this all together,
  we obtain
  \begin{align*}
    u_{a \cup \{\beta\}} \cap u_{a \cup \{\beta'\}} &= u_b[\mb{r}_{m + 1}]
    \cap u_{b'}[\mb{r}_{m+1}] \\
    &= u_b[\mb{r}_{m + 1}] \cap u_{b'}[\mb{r}_{m+1}] \cap u_b[\mb{r}_{n \setminus \{m\}}]
    \cap u_{b'}[\mb{r}_{n \setminus \{m\}}] \\
    &= u_b[\mb{r}_m] \cap u_{b'}[\mb{r}_m] \\
    &= u_a \cap u_a = u_a,
  \end{align*}
  thus completing the proof.
\end{proof}

We now recall a fact which follows from the main result of \cite{hdds}.
In its statement, we use the following notation.

\begin{definition}\label{sigmadefinition}
  Suppose that $\lambda$ is an infinite regular cardinal. Recursively define
  cardinals $\sigma(\lambda, n)$ for $1 \leq n < \omega$ by letting
  $\sigma(\lambda, 1) = \lambda$ and, given $1 \leq n < \omega$, letting
  $\sigma(\lambda, n + 1) = (2^{<\sigma(\lambda, n)})^+$.
\end{definition}

\begin{fact}{\cite[Theorem 2.10]{hdds}} \label{delta_systems_fact}
  Suppose that
  \begin{itemize}
    \item $1 \leq n < \omega$;
    \item $\kappa < \lambda$ are infinite cardinals, $\lambda$ is regular and
    ${<}\kappa$-inaccessible, and $\mu = \sigma(\lambda, n)$;
    \item $c:[\mu]^n \rightarrow 2^{<\kappa}$;
    \item for all $b \in [\mu]^n$, we are given a set $u_b \in [\mathrm{On}]^{<\kappa}$.
  \end{itemize}
  Then there are an $H \in [\mu]^\lambda$ and $k < 2^{<\kappa}$ such that
  \begin{enumerate}
    \item $c(b) = k$ for all $b \in [H]^n$;
    \item $\langle u_b \mid b \in [H]^n \rangle$ is a uniform $n$-dimensional
    $\Delta$-system.
  \end{enumerate}
\end{fact}

In order to motivate these definitions, let us highlight a way in which Fact
\ref{delta_systems_fact} will be employed.

\begin{lemma} \label{compatibility_lemma}
  Suppose that $n$ is a positive integer, $H$ is a set of ordinals, and
  $\langle p_b \mid b \in [H]^n \rangle$ is a sequence of conditions in some $\bb{P}=\mathrm{Add}(\omega,\chi)$
  such that
  \begin{itemize}
    \item there is a fixed $\bar{p}$ such that $\bar{p} = \bar{p}_b$ for all
    $b \in [H]^n$, and
    \item $\langle u(p_b) \mid b \in [H]^n \rangle$ is a uniform
    $n$-dimensional $\Delta$-system.
  \end{itemize}
  Then for all $a, a' \in [H]^n$, if $a$ and $a'$ are aligned, then
  $p_a$ and $p_{a'}$ are compatible in $\bb{P}$.
\end{lemma}

\begin{proof}
  Fix $a,a' \in [H]^n$ such that $a$ and $a'$ are aligned. To show that
  $p_a$ and $p_{a'}$ are compatible, it suffices to show that, for every
  $(\alpha, m) \in \dom(p_a) \cap \dom(p_{a'})$, we have $p_a(\alpha, m) =
  p_{a'}(\alpha, m)$. To this end, fix $(\alpha, m) \in \dom(p_a) \cap
  \dom(p_{a'})$, so $\alpha \in u(p_a) \cap u(p_{a'})$. Since
  $\langle p_b \mid b \in [H]^n \rangle$ is a uniform $n$-dimensional
  $\Delta$-system and $a, a' \in [H]^n$ are aligned, it follows that
  $u(p_a)$ and $u(p_{a'})$ are aligned. There is therefore a single $i < \omega$
  such that $\alpha = u(p_\alpha)(i) = u(p_{\alpha'})(i)$. But then, since
  $\bar{p} = \bar{p}_\alpha = \bar{p}_{\alpha'}$, we have
  $p_a(\alpha, m) = \bar{p}(i, m) = p_{a'}(\alpha, m)$, as desired.
\end{proof}

\subsection{Higher-dimensional coherence and triviality}\label{highercoherencesection}
Our main theorem is an assertion about the derived limits of an inverse system $\mathbf{A}$; just as in \cite{mp} and \cite{SVHDL} (and, indeed, as in all the intervening works cited in the latter), our analysis of these limits will be via their reformulation in terms of multidimensionally coherent
indexed families of functions from subsets of $\omega^2$ to $\mathbb{Z}$. Readers are referred to the second of the aforementioned works for fuller details of this reformulation. We turn now to the relevant conventions and definitions.

Given functions $f,g : \omega \rightarrow \omega$, let $f \leq g$ if and only if
$f(j) \leq g(j)$ for all $j \in \omega$. Let $I(f)$ denote the set
$\left\{(j,k) \in \omega^2 ~ \middle| ~ k \leq f(j)\right\}$; visually, this is the region
below the graph of $f$. Given a sequence $\vec{f} = (f_0, \ldots, f_n)$ of elements of ${^\omega}\omega$, let $\wedge \vec{f}$ denote the greatest lower $\leq$-bound of the functions
$f_0, \ldots, f_n$. For any such sequence and $i \leq n$, let $\vec{f}^i$ denote the sequence of length $n$ obtained by removing the $i^{\mathrm{th}}$ entry of
  $\vec{f}$; in symbols, $\vec{f}^i = (f_0, \ldots, f_{i-1}, f_{i+ 1}, \ldots, f_n)$, sometimes written as $(f_0, \ldots, \hat{f}_i, \ldots, f_n)$. If $\pi$ is a permutation of $(0, \ldots, n)$, then $sgn(\pi)$ denotes the
  \emph{sign} or \emph{parity} of $\pi$, recorded as a $1$ or $-1$.
  The notation $\pi(\vec{f})$ denotes the sequence
  $(f_{\pi(0)}, \ldots, f_{\pi(n)})$.

If $\varphi$ and $\psi$ are partial functions from $\omega^2$ to $\mathbb{Z}$, then the expression $\varphi =^* \psi$ will mean that the set $\{(j,k)\in\mathrm{dom}(\varphi)\cap\mathrm{dom}(\psi)\mid \varphi(j,k)\neq\psi(j,k)\}$ is finite. Implicit in this expression, in other words, are restrictions of $\varphi$ and $\psi$ to their shared domain; a similar convention will apply to sums of such functions below.

\begin{definition} \label{NCOH}
  Fix an $X\subseteq\,^\omega\omega$ and a positive integer $n$ and suppose that
  \[
    \Phi = \left\langle\varphi_{\vec{f}} : I(\wedge \vec{f}) \rightarrow \bb{Z} ~
    \middle| ~ \vec{f} \in X^n\right\rangle
  \]
  is an indexed family of functions.
  \begin{itemize}
    \item $\Phi$ is \emph{alternating} if
    \[
      \varphi_{\pi(\vec{f})} = sgn(\pi) \varphi_{\vec{f}}
    \]
    for every $\vec{f} \in X^n$ and every permutation $\pi$ of
    $(0, \ldots, n-1)$.
    \item $\Phi$ is \emph{$n$-coherent} if it is alternating and if
    \[
      \sum_{i = 0}^n (-1)^i \varphi_{\vec{f}^i} =^* 0
    \]
    for all $\vec{f} \in X^{n+1}$. (As indicated, for readability, here we omit the restrictions
    of the functions in the above expression to the intersection of their domains;
   formally, each $\varphi_{\vec{f}^i}$ should be $\varphi_{\vec{f}^i} \restriction
    I(\wedge \vec{f})$. We will continue this practice below.)
    \item If $n = 1$, then $\Phi$ is \emph{$n$-trivial} (i.e., \emph{1-trivial}) if there exists a $\psi:\omega^2\to\mathbb{Z}$ such that
    \[\psi=^*\varphi_f\]
    for all $f\in X$. If $n>1$, then $\Phi$ is \emph{$n$-trivial} if there exists an alternating family
    \[
      \Psi = \left\langle\psi_{\vec{f}} : I(\wedge \vec{f}) \rightarrow \bb{Z} ~ \middle| ~
      \vec{f} \in X^{n-1}\right\rangle
    \]
    such that
    \[
      \sum_{i = 0}^{n-1} (-1)^i \psi_{\vec{f}^i} =^* \varphi_{\vec{f}}
    \]
    for all $\vec{f} \in X^n$. We term such a $\psi$ or $\Psi$ an
    \emph{$n$-trivialization} of $\Phi$.
  \end{itemize}
  When it is clear from context, we will frequently omit the prefix $n$-\ when speaking of triviality.  Lastly, if $A\subseteq X$, then $\Phi\restriction A$ denotes $\langle\varphi_{\vec{f}}\mid\vec{f}\in A^n\rangle$.
  \end{definition}
\begin{observation}\label{eastofell}
Fix an $\ell\in\omega$ and a family of functions $\Phi=\langle\varphi_{\vec{f}}:I(\wedge\vec{f})\to\mathbb{Z}\mid\vec{f}\in X^n\rangle$. Let $\tilde{\Phi}=\langle\tilde{\varphi}_{\vec{f}}\mid\vec{f}\in X^n\rangle$ where
\[
  \tilde{\varphi}_{\vec{f}}\hspace{.03 cm}(j,k) =
  \begin{cases}
                                   0 & \text{if $j\leq\ell$} \\
                                   \varphi_{\vec{f}}\hspace{.03 cm}(j,k) & \text{if $j>\ell$}
  \end{cases}
\]
for all $(j,k)\in I(\wedge\vec{f})$. Then $\Phi$ is trivial if and only if $\tilde{\Phi}$ is.
\end{observation}
By the following equivalence, our main theorem is, equivalently, a statement about the $n$-triviality of all $n$-coherent families of functions $\Phi=\langle\varphi_{\vec{f}}:I(\wedge\vec{f})\to\mathbb{Z}\mid\vec{f}\in(^\omega\omega)^n\rangle$.
\begin{fact}[\cite{B1,SVHDL}] \label{mostmainfact} For all positive integers $n$, $\mathrm{lim}^n\mathbf{A}=0$ if and only if every $n$-coherent family of functions $\langle\varphi_{\vec{f}}:I(\wedge\vec{f})\to\mathbb{Z}\mid\vec{f}\in(^\omega\omega)^n\rangle$ is trivial.
\end{fact}
Our overall argument's strategy is to arrange this fact's latter condition; in such an approach, homological algebraic considerations appear as essentially external, or preliminary, to our main work. Instrumental  in our forcing arguments, however, will be a more locally finitary characterization of $n$-triviality, one connecting to that of Definition \ref{NCOH} via a long exact sequence of higher derived limits of $\mathbf{A}$-related inverse systems. As noted above, readers primarily interested in those arguments may proceed without danger or delay to Proposition \ref{trivial_equivalence_fact} and continue their reading from there. Its argument is straightforward, algebraic, and occupies roughly the next two pages.

\begin{definition}
The inverse systems \begin{itemize}
\item $\mathbf{A}=(A_f,p_{fg},{^\omega}\omega)$
\item $\mathbf{B}=(B_f,q_{fg},{^\omega}\omega)$
\item $\mathbf{B}/\mathbf{A}=((B/A)_f,r_{fg},{^\omega}\omega)$
\end{itemize}
are defined as follows: $A_f=\bigoplus_{I(f)}\mathbb{Z}$, $B_f=\prod_{I(f)}\mathbb{Z}$, and $(B/A)_f=B_f/A_f$, for all $f$ in ${^\omega}\omega$. For all $f\leq g$ in ${^\omega}\omega$, the bonding maps $p_{fg}:A_g\to A_f$ are simply the projection maps; similarly for the bonding maps of $\mathbf{B}$ and $\mathbf{B}/\mathbf{A}$. These systems assemble in a short exact sequence
\begin{align}\label{ses}0\longrightarrow \,\mathbf{A}\longrightarrow \,\mathbf{B}\longrightarrow \,\mathbf{B}/\mathbf{A}\longrightarrow 0
\end{align}
which induces, in turn, a long exact sequence of derived limits:
\begin{align}\label{les}0\longrightarrow \,\text{lim}\,\mathbf{A}\longrightarrow \text{lim}\,\mathbf{B}\longrightarrow \text{lim}\,\mathbf{B}/\mathbf{A}
\xrightarrow{\partial^0} \text{lim}^1\,\mathbf{A}\longrightarrow \text{lim}^1\,\mathbf{B}\longrightarrow \text{lim}^1\,\mathbf{B}/\mathbf{A}\xrightarrow{\partial^1} \cdots\end{align}
For any $X\subseteq{^\omega}\omega$ we write $\mathbf{A}\restriction X$ for the restriction of $\mathbf{A}$ to the index-set $X$; similarly for $\mathbf{B}$ and $\mathbf{B}/\mathbf{A}$. Observe that for any such $X$, short and long exact sequences just as above exist for $\mathbf{A}\restriction X$, $\mathbf{B}\restriction X$, and $\mathbf{B}/\mathbf{A}\restriction X$.
\end{definition}
The expressions $\mathrm{lim}^n$ denote the \emph{derived limits} of the inverse limit functor $\mathrm{lim}$; they are functors taking, in our context, inverse systems to the category of abelian groups. A standard heuristic for these functors is that, in aggregate, they at least potentially recover the data of an inverse system that the $\mathrm{lim}$ functor alone might lose. The exact sequences above are a main instance of this dynamic: $\mathrm{lim}$ alone applied to the sequence (\ref{ses}) may fail to conserve its exactness, but when applied in combination with the higher derived limits $\mathrm{lim}^n$, as in (\ref{les}), it does transmit exactness, as desired.

Our characterizations of the vanishing of $\mathrm{lim}^n\mathbf{A}$ derive from the isomorphism $\mathrm{lim}^n\,\mathbf{B}/\mathbf{A}\cong\mathrm{lim}^{n+1}\mathbf{A}$ for all $n>0$; the \emph{mod finite} relations of Definition \ref{NCOH}, for example, are an artifact of the modulus $\mathbf{A}$ on the left-hand side of this isomorphism. This isomorphism, in turn, is an effect of following fact within the long exact sequence (\ref{les}).
\begin{lemma}\label{vanishingB} $\mathrm{lim}^n(\mathbf{B}\restriction X)=0$ for all $X\subseteq{^\omega}\omega$ and $n>0$.
\end{lemma}
We argue this fact via more concrete characterizations of $\mathrm{lim}^n(\mathbf{B}\restriction X)$; these are essentially those given by \cite[pages 10-11]{SVHDL}.
\begin{definition}
For $n\geq 0$, the group $\mathrm{lim}^n(\mathbf{B}\restriction X)$ is the cohomology of the cochain complex
\begin{align*}\cdots\xrightarrow{d^{n-1}} K^n(\mathbf{B}\restriction X)\xrightarrow{d^n}K^{n+1}(\mathbf{B}\restriction X)\xrightarrow{d^{n+1}}\cdots\end{align*}
where $K^n(\mathbf{B}\restriction X)$ denotes the subgroup of $$\prod_{\vec{f}\in X^{n+1}}B_{\wedge\vec{f}}$$ whose elements $c$ satisfy $$c(\pi(\vec{f}))=sgn(\pi)c(\vec{f})$$
for all $\vec{f}\in X^{n+1}$ and permutations $\pi$ of $(0,\dots,n)$. The differentials $d^n$ are defined as usual: for any $c\in K^n(\mathbf{B}\restriction X)$,
\begin{align}\label{sum1}
d^n c(\vec{f})=\displaystyle\sum_{i=0}^{n+1}(-1)^i \left(c(\vec{f}^i)\restriction I(\wedge \vec{f})\right)
\end{align}
for each $\vec{f}\in X^{n+2}$.

$\mathrm{lim}^n(\mathbf{A}\restriction X)$ and $\mathrm{lim}^n(\mathbf{B}/\mathbf{A}\restriction X)$ are defined analogously.
\end{definition}

\begin{proof}[Proof of Lemma \ref{vanishingB}]
Fix an $X\subseteq{^\omega}\omega$ and an $n>0$ along with a $c\in K^n(\mathbf{B}\restriction X)$ for which $d^n c=0$. We will define a $b\in K^{n-1}(\mathbf{B}\restriction X)$ with $d^{n-1}b=c$. To that end, for each $x\in\bigcup_{f\in X} I(f)$ fix an $f_x\in X$ such that $x\in I(f_x)$. Then, for each $\vec{f}\in X^n$ and $x\in I(\wedge\vec{f})$, let $$b(\vec{f})(x)=(-1)^n c(f_0,\dots,f_{n-1},f_x)(x)\text{.}$$ For all $\vec{f}\in X^{n+1}$ and $x\in I(\wedge\vec{f})$, we then have
\begin{align*}
d^{n-1} b(\vec{f})(x) & = \sum_{i=0}^n (-1)^i b(\vec{f}^i)(x) \\
& = \sum_{i=0}^n (-1)^{n+i} c(f_0,\dots,\hat{f}_i,\dots,f_n,f_x)(x) \\
& = (-1)^n d^n c(f_0,\dots,f_n,f_x)(x) - (-1)^{2n+1} c(\vec{f})(x) \\
& = c(\vec{f})(x)\text{,}
\end{align*}
as desired.
\end{proof}

As noted, together with the $X$-indexed variants of the long exact sequence (\ref{les}), Lemma \ref{vanishingB} has as consequences isomorphisms
$$\partial^n:\mathrm{lim}^n(\mathbf{B}/\mathbf{A}\restriction X)\xrightarrow{\;\cong\;}\mathrm{lim}^{n+1}(\mathbf{A}\restriction X)$$
for each $n>0$, as well as the isomorphism
$$\partial^0:\frac{\mathrm{lim}(\mathbf{B}/\mathbf{A}\restriction X)}{\mathrm{im}(\mathrm{lim}(\mathbf{B}\restriction X))}\xrightarrow{\;\cong\;}\mathrm{lim}^1(\mathbf{A}\restriction X)\text{.}$$
As shown in \cite{SVHDL}, for all $n\geq 0$ these isomorphisms $\partial^n$ may be defined via the following procedure: \begin{enumerate}
\item To define $\partial^n[c]$, fix a cocycle $c\in K^n(\mathbf{B}/\mathbf{A}\restriction X)$ representing the cohomology class $[c]$.
\item Fix then a $\Phi=\langle\varphi_{\vec{f}}: I(\wedge\vec{f})\to\mathbb{Z}\mid\vec{f}\in X^{n+1}\rangle$ representing $c$ in the sense that each $\varphi_{\vec{f}}$ falls in the $A_{\vec{f}}$-coset $c(\vec{f})$. Observe that such a $\Phi$ may be chosen to be alternating, and that in this case it will be an $(n+1)$-coherent family of functions. (In fact, for all $n\geq 0$ the left-hand side of the above isomorphisms is naturally viewed as the quotient of $(n+1)$-coherent families of functions indexed by $X$ by the $(n+1)$-trivial families of functions indexed by $X$, as the reader may verify.)
\item Let $\partial^n[c]$ be the cohomology class of the cocycle $\mathsf{d}^n\Phi\in K^{n+1}(\mathbf{A}\restriction X)$, where $$\mathsf{d}^n\Phi(\vec{f})=\sum_{i=0}^{n+1} (-1)^i \varphi_{\vec{f}^i}$$
for all $\vec{f}\in X^{n+2}$.
\end{enumerate}
That this procedure defines an isomorphism implies that an $(n+1)$-coherent family $\Phi$ is $(n+1)$-trivial if and only if $\mathsf{d}^n\Phi$ is a coboundary in $K^{n+1}(\mathbf{A}\restriction X)$, i.e., if and only if there exists a family of \emph{finitely supported} functions $\langle\psi_{\vec{f}}: I(\wedge\vec{f})\to\mathbb{Z}\mid\vec{f}\in X^{n+1}\rangle$ such that $$\mathsf{d}^n\Phi(\vec{f})=\sum_{i=0}^{n+1} (-1)^i \psi_{\vec{f}^i}$$
for all $\vec{f}\in X^{n+2}$. By way of these observations, together with Observation \ref{eastofell}, we arrive to our second criterion for $n$-triviality:
\begin{proposition} \label{trivial_equivalence_fact}
  Fix $X \subseteq{^\omega}\omega$ and a positive integer $n$ and let
  $\Phi = \langle \varphi_{\vec{f}} \mid \vec{f} \in X^n \rangle$ be an
  $n$-coherent family of functions. Then $\Phi$ is trivial if and only if there exists an
  $\ell < \omega$ and an alternating family of finitely supported functions
  $\Psi = \langle \psi_{\vec{f}}:I(\wedge\vec{f}) \rightarrow \bb{Z} \mid
  \vec{f} \in X^n \rangle$ such that
  \[
    \sum_{i=0}^n (-1)^i \varphi_{\vec{f}^i}(j,k) = \sum_{i=0}^n (-1)^i \psi_{\vec{f}^i}(j,k)
  \]
 for all $\vec{f} \in
  X^{n+1}$ and all $(j,k) \in I(\wedge\vec{f})$ with $j > \ell$.
\end{proposition}

When there is a possibility of confusion, we will refer to a $\Psi$ as above as a \emph{type II trivialization} and a $\Psi$ or $\psi$ as in Definition \ref{NCOH} as a \emph{type I trivialization}. By and large, however, these two sorts of trivializations correspond to two distinct phases of our argument; in particular, the trivializations under discussion in Sections \ref{n1section} and \ref{propagatingsection} are all of type I, while those under discussion in Sections \ref{defining_trivializations_section} and \ref{highernsection} are of type II.

\section{The case of $n=1$}\label{n1section}
We now argue the base case of our main theorem. The main result of this section is
essentially due to Kamo \cite{kamo}. In fact, the result in \cite{kamo} is superior
to the one presented here in that Kamo proves that $\lim^1 \mathbf{A} = 0$ in any
extension obtained by adding $\omega_2$-many Cohen reals, whereas our hypothesis
is that we have added at least $(\beth_1^+)^V$-many Cohen reals. Our reason for
presenting this slightly suboptimal proof is simply that many of the ideas of
the proof of the general case appear here in a significantly simplified setting;
this section therefore serves as an introduction to some of the techniques and
ideas that will make an appearance in a more complicated guise later in the paper.

\begin{theorem} \label{1d_theorem}
  Let $\bb{P} = \mathrm{Add}(\omega, \chi)$ for a cardinal $\chi > \beth_1$.
  The following then holds in $V^{\bb{P}}$: For any set $X \subseteq {^\omega}\omega$
  containing at least $(\beth_1^+)^V\!$-many of the Cohen reals added by $\bb{P}$,
  every $1$-coherent family $\Phi = \langle \varphi_f \mid f \in X\rangle$ indexed by $X$
  is trivial.
\end{theorem}

\begin{proof}
  Fix a condition $p \in \bb{P}$ and $\bb{P}$-names $\dot{X}$ and
  $\dot{\Phi} = \langle \dot{\varphi}_{\dot{f}} \mid \dot{f} \in \dot{X} \rangle$ such that
  \begin{itemize}
    \item $p \Vdash ``|\{\alpha < \chi \mid \dot{f}_\alpha \in \dot{X}\}| \geq
    (\beth_1^+)^V"$, and
    \item $p \Vdash ``\dot{\Phi} \text{ is a 1-coherent family}"$.
  \end{itemize}
  We will produce a condition $q \leq p$ forcing $\dot{\Phi}$ to be trivial.

  Begin by letting $Y$ be the set of $\alpha < \chi$ such that there is a condition
  $p_\alpha \leq p$ such that $p_\alpha \Vdash ``\dot{f}_\alpha \in \dot{X}"$; observe that $|Y| \geq \beth_1^+$ by assumption.
  For each $\alpha \in Y$, fix such a condition $p_\alpha$. Since $\bb{P}$ is
  $\beth_1^+$-Knaster, there exists a set $Y' \subseteq Y$ of size $\beth_1^+$
  such that $\{p_\alpha \mid \alpha \in Y'\}$ consists of pairwise compatible
  conditions.

  For each $(\alpha, \beta) \in [Y']^2$, fix a condition $q_{\alpha, \beta}$
  extending both $p_\alpha$ and $p_\beta$ and deciding the value of
  $\{(j,k) \in I(\dot{f}_\alpha, \dot{f}_\beta) \mid
  \dot{\varphi}_{\dot{f}_\alpha}(j,k) \neq \dot{\varphi}_{\dot{f}_\beta}(j,k)\}$
  to be equal to some set $\mathsf{e}_{\alpha, \beta} \in [\omega \times \omega]^{<\omega}$.
  By extending $q_{\alpha, \beta}$ if necessary, we may assume that
  $\{\alpha, \beta\} \subseteq u(q_{\alpha, \beta})$.
  Let $u_{\alpha, \beta} = u(q_{\alpha, \beta})$.

  By Fact \ref{delta_systems_fact}, there exists a set $H \in [Y']^{\aleph_1}$
  and a $\bar{q}$, $\mathsf{e}$, and $i^*$ such that
  \begin{itemize}
    \item $\langle u_{\alpha, \beta} \mid (\alpha, \beta) \in [H]^2 \rangle$
    is a uniform 2-dimensional $\Delta$-system;
    \item $(\bar{q}_{\alpha, \beta}, \mathsf{e}_{\alpha, \beta}) = (\bar{q},
    \mathsf{e})$ for all $(\alpha, \beta) \in [H]^2$;
    \item $\beta = u_{\alpha, \beta}(i^*)$ for all $(\alpha, \beta) \in [H]^2$.
  \end{itemize}

  By shrinking $H$ if necessary, we may assume that $\otp(H) = \omega_1$. Now let
  $\langle r_{\mathbf{m}}\mid \mathbf{m}\subseteq 2\rangle$ witness that $\langle
  u_{\alpha, \beta} \mid (\alpha, \beta) \in [H]^2 \rangle$ is a uniform
  2-dimensional $\Delta$-system, and let $\langle u_\alpha \mid \alpha \in H \rangle$
  and $u_\emptyset$ be as given by Lemma \ref{extension_lemma}. For each $\alpha \in H$,
  define a condition $q_\alpha \in \bb{P}$ by choosing a $\beta \in H \setminus (\alpha + 1)$
  and letting $q_\alpha = q_{\alpha, \beta} \restriction (u_\alpha \times \omega)$.
  We claim that this definition is independent of our choice of $\beta$. Indeed,
  suppose that $\beta < \beta'$ are elements of $H \setminus (\alpha + 1)$.
  Then $u_\alpha = u_{\alpha, \beta} \cap u_{\alpha, \beta'} =
  u_{\alpha, \beta}[\mb{r}_1] = u_{\alpha, \beta'}[\mb{r}_1]$, hence if
  $(\delta, m) \in \dom(q_{\alpha, \beta}) \cap (u_\alpha \times \omega)$
  then there is an $i \in \mb{r}_1$ such that $u_{\alpha, \beta}(i) = \delta =
  u_{\alpha, \beta'}(i)$. Since $\bar{q}_{\alpha, \beta} = \bar{q} =
  \bar{q}_{\alpha, \beta'}$, it follows that
  $q_{\alpha, \beta}(\delta, m) = \bar{q}(i, m) = q_{\alpha, \beta'}(\delta, m)$,
  so $q_{\alpha, \beta} \restriction (u_\alpha \times \omega) \subseteq
  q_{\alpha, \beta'} \restriction (u_\alpha \times \omega)$. A symmetric argument
  yields the reverse inclusion, showing that our definition of $q_\alpha$ is indeed
  independent of our choice of $\beta$. Observe that
  $q_\alpha = \bigcap \{q_{\alpha, \beta} \mid \beta \in H \setminus (\alpha + 1)\}$; as each $q_{\alpha, \beta}$ extends $p_\alpha$, it follows that
  $q_\alpha \leq p_\alpha$ and hence that $q_\alpha \Vdash ``\dot{f}_\alpha \in \dot{X}"$.

  Similarly, define a condition $q_\emptyset$ by choosing $(\alpha, \beta)
  \in [H]^2$ and letting $q_\emptyset = q_{\alpha, \beta} \restriction
  (u_\emptyset \times \omega)$. By an argument exactly as in the previous paragraph, this definition is independent of our choice of $(\alpha, \beta)$. Note
  that $q_\emptyset = \bigcap \{q_{\alpha, \beta} \mid (\alpha, \beta) \in [H]^2\}$; in consequence, since each $q_{\alpha, \beta}$ extends $p$, we have $q_\emptyset \leq p$.

We claim that $q_\emptyset$ forces that $\dot{\Phi}$ is trivial. Let
  $\dot{A}$ be a $\bb{P}$-name for $\{\alpha \in H \mid q_\alpha \in \dot{G}\}$,
  where $\dot{G}$ is the canonical $\bb{P}$-name for the generic filter.

  \begin{claim} \label{unbounded_1d_claim}
    $q_\emptyset \Vdash ``|\dot{A}| = \aleph_1"$.
  \end{claim}

  \begin{proof}
    It suffices to show for each $\eta \in H$ that the set
    $\{q_\alpha \mid \alpha \in H \setminus \eta\}$ is pre-dense below
    $q_\emptyset$. To this end, fix such an $\eta \in H$ and an $r \leq q_\emptyset$.
    We desire an $\alpha \in H \setminus \eta$ such that $q_\alpha$ is compatible
    with $r$. Since $\langle u_\alpha \mid \alpha \in H \setminus \eta \rangle$ is
    an infinite $\Delta$-system with root $u_\emptyset$, and since $u_r$ is finite,
    there exists an $\alpha \in H \setminus \eta$ such that $u_\alpha \setminus u_\emptyset$
    is disjoint from $u_r$. Decompose $q_\alpha$ as $q_\emptyset \cup (q_\alpha \restriction
    (u_\alpha \setminus u_\emptyset) \times \omega)$ and observe that $r \leq q_\emptyset$
    and $\dom(r) \cap ((u_\alpha \setminus u_\emptyset) \times \omega) = \emptyset$.
    It follows that $r$ and $q_\alpha$ are compatible, as desired.
  \end{proof}

  Now recall that $\beta = u_{\alpha,
  \beta}(i^*)$ for all $(\alpha, \beta) \in [H]^2$. Notice that $i^* \notin \mb{r}_1$, since the alternative would imply that $\beta \in u_\alpha =
  u_{\alpha, \beta}[\mb{r}_1]$ for all
  $(\alpha, \beta) \in [H]^2$, contradicting the fact that $u_\alpha$ is finite.
  Hence $\beta \in u_{\alpha, \beta} \setminus u_\alpha$ for all $(\alpha, \beta)
  \in [H]^2$.
  Let $\ell$ be the least natural number $j$ such that
  \begin{itemize}
    \item $\mathsf{e} \subseteq j \times \omega$ and;
    \item $\{j' \mid (i^*, j') \in \dom(\bar{q})\} \subseteq j$.
  \end{itemize}
  We then have
  $\{j' \mid (\beta, j') \in \dom(q_{\alpha, \beta})\} \subseteq \ell$ for each $(\alpha, \beta) \in [H]^2$.

  \begin{claim} \label{agreement_claim}
    $q_\emptyset$ forces that
    \[
      \{(j,k) \in I(\dot{f_\alpha}, \dot{f_{\alpha'}}) \mid \dot{\varphi}_{\dot{f}_\alpha}
      (j,k) \neq \dot{\varphi}_{\dot{f}_{\alpha'}}(j,k)\} \subseteq \ell \times \omega.
    \]
    for all $\alpha, \alpha' \in \dot{A}$.
  \end{claim}

  \begin{proof}
    If not, then there exist an $r \leq q_\emptyset$, a pair of ordinals $\alpha < \alpha'$ in $H$,
    and a $(j,k) \in \omega \times \omega$ such that
    \begin{itemize}
      \item $r \leq q_\alpha, q_{\alpha'}$;
      \item $j \geq \ell$;
      \item $r$ forces that $(j,k)$ is in $I(\dot{f_\alpha}, \dot{f_{\alpha'}})$
      and that $\dot{\varphi}_{\dot{f}_\alpha}(j,k) \neq \dot{\varphi}_{\dot{f}_{\alpha'}}
      (j,k)$.
    \end{itemize}
    Both $\langle u_{\alpha, \beta} \mid \beta \in H\backslash(\alpha+1) \rangle$ and
    $\langle u_{\alpha', \beta} \mid \beta \in H\backslash(\alpha+1) \rangle$ are
    infinite $\Delta$-systems with roots $u_\alpha$ and $u_{\alpha'}$, respectively; as $u_r$ is finite, there therefore exists a $\beta \in H$ such
    that both $u_{\alpha, \beta} \setminus u_\alpha$ and $u_{\alpha', \beta}
    \setminus u_{\alpha'}$ are disjoint from $u_r$.

    By Lemma
    \ref{compatibility_lemma}, the conditions $q_{\alpha, \beta}$ and $q_{\alpha', \beta}$ are compatible. Observe also that $q_{\alpha, \beta} = q_\alpha \cup (q_{\alpha,
    \beta} \restriction (u_{\alpha, \beta} \setminus u_\alpha) \times \omega)$.
    Since $r \leq q_\alpha$ and $\dom(r) \cap ((u_{\alpha, \beta} \setminus
    u_\alpha) \times \omega) = \emptyset$, the conditions $r$ and $q_{\alpha, \beta}$
    are compatible. Similarly, $r$ and $q_{\alpha', \beta}$ are compatible, and
    therefore $r^* = r \cup q_{\alpha, \beta} \cup q_{\alpha', \beta}$ is a condition
    in $\bb{P}$. Notice also that $\beta \notin u_r$. By the paragraph preceding
    Claim \ref{agreement_claim} and the fact that $j \geq \ell$, it follows that $(\beta, j) \notin
    \dom(r^*)$, so we may extend $r^*$ to a condition $r^{**}$ such that
    $(\beta, j) \in \dom(r^{**})$ and $r^{**}(\beta, j) = k$. In particular, $r^{**}$ will force that $(j,k)$ is in $I(\dot{f}_\beta)$.

    Recall that $j \geq \ell$ implies $(j,k) \notin \mathsf{e}$. Therefore, since it extends both $q_{\alpha, \beta}$ and $q_{\alpha', \beta}$ and
    forces $(j,k)$ to be in $I(\dot{f}_\alpha, \dot{f}_{\alpha'}, \dot{f}_\beta)$, the condition $r^{**}$ will force
    \[
      \text{``}\dot{\varphi}_{\dot{f}_\alpha}(j,k) = \dot{\varphi}_{\dot{f}_\beta}(j,k)
      = \dot{\varphi}_{\dot{f}_{\alpha'}}(j,k)\text{''},
    \]
    contradicting the fact that $r^{**} \leq r$ and $r$ forces
    $\text{``}\dot{\varphi}_{\dot{f}_\alpha}(j,k) \neq \dot{\varphi}_{\dot{f}_\alpha'}(j,k)\text{''}$.
  \end{proof}
  Now let $G$ be $\bb{P}$-generic over $V$ with $q_\emptyset \in G$. Let
  $\Phi = \langle \varphi_f \mid f \in X \rangle$ and $A$ denote the realizations in $V[G]$ of
  $\dot{\Phi}$ and $\dot{A}$, respectively. Define a function
  $\psi : \omega \times \omega \rightarrow \bb{Z}$ as follows. For any
  $(j,k) \in \omega \times \omega$ with $j \geq \ell$, if $(j,k) \in I(f_\alpha)$ for some $\alpha \in A$
then let $\psi(j,k) = \varphi_{f_\alpha}(j,k)$ (by Claim \ref{agreement_claim},
  this definition is independent of our choice of $\alpha$). In all other cases, let
  $\psi(j,k) = 0$.

  We claim that $\psi$ witnesses that $\Phi$ is trivial. Assume for contradiction that it does not, so that for some $f \in X$ the set $E_f := \{(j,k) \in I(f) \mid \varphi_f(j,k)
  \neq \psi(j,k)\}$ is infinite. Since $I(f) \cap (\ell \times \omega)$ is finite,
  the set $E_f^* = E_f \cap ([\ell, \omega) \times \omega)$ is then infinite and there are infinitely many $j < \omega$ for which
  $E_f^* \cap (j \times \omega) \neq \emptyset$. As $\bb{P}$ has the countable chain condition, $E_f \in
  V[G_W]$ for some countable set $W \subseteq \chi$ in $V$. Fix $\alpha \in A \setminus W$. By genericity, $I(f_\alpha) \cap E_f^*$ is infinite. It follows from our definition of $\psi$ that $\psi(j,k) = \varphi_{f_\alpha}(j,k)$ for all $(j,k) \in I(f_\alpha) \cap E_f^*$. Hence $\varphi_f \restriction (
  I(f_\alpha) \cap E_f^*) =^* \psi \restriction (I(f_\alpha) \cap E_f^*)$, by the coherence of $\Phi$,
  contradicting the definition of $E_f^*$ and the fact that $I_{f_\alpha} \cap
  E_f^*$ is infinite. This shows that $q_\emptyset$ forces $\dot{\Phi}$
  to be trivial, concluding the proof.
\end{proof}
Already in the above proof the broader contours of our more general argument are legible. In particular, observe that comparatively few of the Cohen reals in $X$, namely just those in the size-$\aleph_1$ set $\{f_\alpha\mid\alpha\in A\}$, played any essential role in the definition of the trivialization $\psi$. Similarly, in what follows we will construct trivializations of higher-dimensional coherent families $\Phi_n=\langle\varphi_{\vec{f}}\mid\vec{f}\in X^n\rangle$ by first trivializing those families over small indexing subsets $A$ of $X$. More precisely, as outlined in our introduction, the arguments for the higher-$n$ versions of Theorem \ref{1d_theorem} will split into two main phases: in the first, repeated application of principles like Fact \ref{delta_systems_fact} determine trivializations of small subfamilies $\Phi_n\restriction A$ of $\Phi_n$; in the second, inductive hypotheses propagate their triviality back out to the entirety of $\Phi_n$. We describe the mechanics of these two phases in reverse order in the next two sections; we then apply these descriptions in the proof of our main theorem, Theorem \ref{maintheorem}.

\section{Propagating trivializations}\label{propagatingsection}

In this section we show how the triviality of restrictions of $n$-coherent families
to domains constaining sufficiently many Cohen reals implies the triviality of the
entire families. We first introduce a slight technical variation of $n$-coherent families.

\begin{definition}
  Fix an $X \subseteq {^\omega}\omega$, a function $h \in {^\omega}\omega$, and
  a positive integer $n$. We say that a family of functions $\Phi = \langle \varphi_{\vec{f}}
  : I(\wedge \vec{f}) \cap I(g) \rightarrow \bb{Z} \mid \vec{f} \in X^n \rangle$
  is \emph{$n$-coherent below $g$} if it satisfies the first two bullet points of
  Definition \ref{NCOH}, the only difference being that in this case the domain
  of $\varphi_{\vec{f}}$ is $I(\wedge \vec{f}) \cap I(g)$ rather than
  $I(\wedge \vec{f})$. We say that such a family is \emph{$n$-trivial below $g$}
  if there is a $\psi$ or $\Psi$ as in the third bullet point of Definition
  \ref{NCOH}, again with the only difference being that, in case $n = 1$,
  we have $\psi:I(g) \rightarrow \bb{Z}$, and in case $n > 1$, we have
  $\psi_{\vec{f}}:I(\wedge \vec{f}) \cap I(g) \rightarrow \bb{Z}$ for all
  $\vec{f} \in X^{n-1}$.
\end{definition}

The following proposition is a simple observation but will be necessary in the
arguments of this section.

\begin{proposition} \label{below_prop}
  Suppose that $X \subseteq {^\omega}\omega$, $h \in {^\omega}\omega$, $n$ is a positive
  integer, and every $n$-coherent family of functions indexed by $X^n$ is trivial.
  Then every $n$-coherent family of functions below $g$ indexed by $X^n$ is trivial
  below $g$.
\end{proposition}

\begin{proof}
  Let $\Phi = \langle \varphi_{\vec{f}} : I(\wedge \vec{f}) \cap I(g) \rightarrow
  \bb{Z} \mid \vec{f} \in X^n \rangle$ be $n$-coherent below $g$. Define an
  $n$-coherent family $\Phi^* = \langle \varphi_{\vec{f}}^* : I(\wedge \vec{f})
  \rightarrow \bb{Z} \mid \vec{f} \in X^n \rangle$ by letting
  $\varphi_{\vec{f}}^*(j,k) = \varphi_{\vec{f}}(j,k)$ for all $(j,k) \in
  I(\wedge \vec{f}) \cap I(g)$ and $\varphi_{\vec{f}}^*(j,k) = 0$ for all
  $(j,k) \in I(\wedge \vec{f}) \setminus I(g)$. It is easily verified that
  $\Phi^*$ is $n$-coherent. By assumption, $\Phi^*$ is trivial, as witnessed by
  a single function $\psi^*$ if $n = 1$ or a family $\Psi^* = \langle \psi^*_{\vec{f}}
  \mid \vec{f} \in X^{n-1} \rangle$ if $n > 1$. If $n = 1$, then the function
  $\psi := \psi^* \restriction I(g)$ witnesses that $\Phi$ is trivial below $g$,
  and if $n > 1$, then the family $\Psi := \langle \psi_{\vec{f}}^* \restriction
  I(\wedge \vec{f}) \cap I(g) \mid \vec{f} \in X^{n-1} \rangle$ witnesses that
  $\Phi$ is trivial below $g$, as desired.
\end{proof}

For motivation, we now begin with the $n=2$ case of Theorem \ref{maintheorem}; assume the aforementioned ``first phase'' of its proof completed. More precisely, this assumption takes the following form: let $\lambda_1=\beth_1^{+}$ and let $\lambda_2=\sigma(\lambda_1^{+},5)$ (see again Definition \ref{sigmadefinition} for the expression $\sigma(\,\cdot\,,\,\cdot\,)$). By the arguments of Section \ref{highernsection}, for any $\chi\geq\lambda_2$, the following will hold in $V^{\mathrm{Add}(\omega,\chi)}$: \emph{if
\begin{itemize}
\item $X\subseteq\,^\omega\omega$ contains at least $\lambda_2$-many of the Cohen reals added by $\mathrm{Add}(\omega,\chi)$, and
\item $\Phi=\langle\varphi_{f,g}\mid (f,g)\in X^2\rangle$ is $2$-coherent,
\end{itemize}
then there exists an $A\subseteq X$ such that
\begin{itemize}
\item $A$ contains $\lambda_1$-many of the Cohen reals added by $\mathrm{Add}(\omega,\chi)$, and
\item $\Phi\restriction A$ is trivial.
\end{itemize}}
\noindent Now let $G$ be $\mathrm{Add}(\omega,\chi)$-generic over $V$ and work in $V[G]$.
By assumption, there exists an $A\subseteq\chi$ indexing $\lambda_1$-many Cohen
reals $f_\alpha$ $(\alpha\in A)$ and a family $\langle\tau_\alpha: I(f_\alpha)\to\mathbb{Z}\mid
\alpha\in A\rangle$ trivializing $\Phi\restriction A$. We propagate this trivialization
to all of $\Phi$ as follows: for all $f\in X$ and $\alpha\in A$, let
\[
  \varsigma_\alpha^f=\varphi_{\alpha, f} +\tau_\alpha\,.
\]
(Here and below, in subscripts we will tend to abbreviate Cohen reals by their indices;
the second $\Phi\restriction A$ just above, which denotes $\langle\varphi_{\alpha,\beta}
\mid(\alpha,\beta)\in A^2\rangle$, is a related minor abuse. We note lastly that,
in keeping with our restriction-conventions, the domain of $\varsigma_\alpha^f$
should be understood to be $\mathrm{dom}(\varphi_{\alpha,f})\cap\mathrm{dom}(\tau_\alpha)$.)

\begin{claim}
  For each $f\in\,X$ the family $C_1^f:=\langle\varsigma^f_\alpha:I(f\wedge f_\alpha)
  \to\mathbb{Z}\mid \alpha\in A\rangle$ is coherent below $f$.
\end{claim}
\begin{proof}
  This follows from the fact that, for all $\alpha$ and $\beta$ in $A$, we have
  \[
    \varsigma^f_\beta-\varsigma^f_\alpha=\varphi_{\beta, f} +
    \tau_\beta-\varphi_{\alpha, f} -\tau_\alpha=^{*}\varphi_{\beta, f}-
    \varphi_{\alpha, f}+\varphi_{\alpha, \beta}=^{*} 0\,,
  \]
  where the first $=^*$ follows from the fact that $\langle \tau_\alpha \mid \alpha \in A
  \rangle$ trivializes $\Phi \restriction A$ and hence $\tau_\beta - \tau_\alpha
  =^* \varphi_{\alpha, \beta}$, the second $=^*$ follows from the 2-coherence of
  $\Phi$.
\end{proof}
As $A$ contains more than $\beth_1$-many Cohen reals,  Theorem \ref{1d_theorem}
and Proposition \ref{below_prop} imply that each such family $C_1^f$ admits a trivialization
 $\tau^f : I(f) \rightarrow \bb{Z}$.
\begin{claim} The family $T_1:=\langle\tau^f\mid f\in X\rangle$ trivializes $\Phi$.
\end{claim}
\begin{proof} Suppose for contradiction that it did not. Then for some $f,g\in X$ and infinite $E \subseteq I(f\wedge g)$, \begin{align*}\tau^g(j,k)-\tau^f(j,k)\neq \varphi_{f,g}(j,k)\end{align*} for all $(j,k)\in E$. As $\mathrm{Add}(\omega,\chi)$ has the countable chain condition, there exists a $W\in [\chi]^{\aleph_0}$ such that $E\in V[G_W]$. By genericity, for any $\beta\in A\backslash W$ the domain $I(f_\beta)$ then has infinite intersection with $E$. However,
\[
  \tau^g-\tau^f=^{*} \varsigma^g_\beta-\varsigma^f_\beta = \varphi_{\beta, g}+\tau_\beta-\varphi_{\beta, f}-\tau_\beta =^{*}\varphi_{f,g}\,,
\]
where, as indicated, the equalities should each be read as applying over the restricted domain $I(f\wedge g\wedge f_\beta)$. It follows that $\tau^g(j,k)-\tau^f(j,k)\neq \varphi_{f,g}(j,k)$ for only finitely many $(j,k)\in E\cap I(f_\beta)$, contradicting our assumption.
\end{proof}
This family $T_1$ is the propagation of the trivialization $\langle\tau_\alpha\mid\alpha\in A\rangle$ to all of $X$ which we had desired. The above technique generalizes, but entails, unsurprisingly, more steps in the cases of higher $n$. We precede its generalized description with a proof-free sketch of the $n=4$ case, simply to better indicate these steps' shape. All coherent and trivializing families below should be understood to be alternating.

\begin{example}\label{ex3} Structuring our argument is an increasing sequence of cardinals $\lambda_n$; see the paragraph preceding Theorem \ref{maintheorem} for their precise definition. The case of $n=4$ begins with a forcing $\mathbb{P}=\mathrm{Add}(\omega,\chi)$ for some $\chi\geq\lambda_4$; as before, we will work in $V^{\mathbb{P}}$.
Let $X\subseteq\,^\omega\omega$ contain at least $\lambda_4$-many of the Cohen reals added by $\mathbb{P}$ and let $\Phi=\langle\varphi_{\vec{f}}\mid \vec{f}\in X^4\rangle$ be $4$-coherent. The arguments of Section \ref{highernsection} will furnish us with an $A\in [\chi]^{\lambda_3}$ and a $$T^A_3=\langle\tau_{\alpha\beta\gamma}\mid (\alpha,\beta,\gamma)\in A^3\rangle$$ trivializing $\Phi\restriction A$.
For each $f\in X$ and $(\alpha,\beta,\gamma)\in A^3$ let $$\varsigma^f_{\alpha\beta\gamma}=\varphi_{\alpha\beta\gamma f}+\tau_{\alpha\beta\gamma}\,.$$
This defines a 3-coherent family $C^f_3=\langle\varsigma^f_{\alpha\beta\gamma}\mid (\alpha,\beta,\gamma)\in A^3\rangle$ below $f$. By the $n=3$ case of Theorem \ref{maintheorem}
and Proposition \ref{below_prop}, for each $f\in X$ there exist trivializations $T^f_2=\langle\tau^f_{\alpha\beta}\mid (\alpha,\beta)\in A^2\rangle$ of $C^f_3$. Using these functions, define for each $(f,g)\in X^2$ the 2-coherent family $C^{fg}_2=\langle\varsigma^{fg}_{\alpha\beta}\mid (\alpha,\beta)\in A^2\rangle$ below
$f \wedge g$ via the assignments
$$\varsigma^{fg}_{\alpha\beta}=\varphi_{\alpha\beta fg}-\tau^g_{\alpha\beta}+\tau^f_{\alpha\beta}\,.$$
By the $n=2$ case of Theorem \ref{maintheorem} and Proposition \ref{below_prop}, for each $(f,g)\in X^2$ there then exist trivializations $T^{fg}_1=\langle\tau^{fg}_{\alpha}\mid \alpha\in A\rangle$ of $C^{fg}_2$. Using these functions, define for each $(f,g,h)\in X^3$ the 1-coherent family $C^{fgh}_1=\langle\varsigma^{fgh}_{\alpha}\mid (\alpha,\beta)\in A^2\rangle$ below
$f \wedge g \wedge h$ via the assignments
$$\varsigma^{fgh}_{\alpha}=\varphi_{\alpha fgh}+\tau^{gh}_{\alpha}-\tau^{fh}_{\alpha}+\tau_{\alpha}^{fg}\,.$$
By Theorem \ref{1d_theorem} and Proposition \ref{below_prop}, for each $(f,g,h)\in X^3$ there exists a trivialization $\tau^{fgh}$ of $C_1^{fgh}$. As above, we then conclude by observing that the collection $\langle\tau^{fgh}\mid (f,g,h)\in X^3\rangle$ trivializes $\Phi$, as desired.
\end{example}

Now, within the larger context of our inductive argument, we summarize the general case. Fix an integer $n>1$ and assume the $n^\mathrm{th}$ instance of our inductive hypothesis, namely, that \emph{for any $j<n$ and cardinal $\chi$, any $j$-coherent family of functions $\Phi=\langle\varphi_{\vec{f}}\mid\vec{f}\in X^j\rangle$ in $V^{\mathrm{Add}(\omega,\chi)}$ whose index-set $X$ contains at least $\lambda_j$-many of the Cohen reals added by $\mathbb{P}=\mathrm{Add}(\omega,\chi)$ is trivial}. We show how, combined with the arguments of Section \ref{highernsection}, the $n^{\mathrm{th}}$ instance of our inductive hypothesis implies the $(n+1)^{\mathrm{st}}$ instance. For $\chi<\lambda_n$ this implication is trivial. Therefore fix $\chi\geq\lambda_n$ and, working in $V^{\mathrm{Add}(\omega,\chi)}$, fix an $n$-coherent family $\Phi=\langle\varphi_{\vec{f}}\mid\vec{f}\in X^n\rangle$ whose index-set $X$ contains at least $\lambda_n$-many of the Cohen reals added by $\mathbb{P}=\mathrm{Add}(\omega,\chi)$. We introduce the organizing notations $\mathscr{T}_k^n$ and $\mathscr{C}_k^n$ and show by the following sequence of steps that $\Phi$ is trivial:
\begin{enumerate}
\item Fix an $n$-coherent $\Phi$ as above. The arguments of Section \ref{highernsection} secure for us a set $A\in [\chi]^{\lambda_{n-1}}$ and a $\mathscr{T}_1^{n}$ such that $\mathscr{T}_1^{n}$ trivializes $\Phi\restriction A$.
\item $\mathscr{T}_1^{n}$ is the first in a sequence of families of functions $$\mathscr{T}_k^{n}=\langle\tau_{\vec{\alpha}}^{\vec{f}}\mid \vec{f}\in X^{k-1}\text{ and }\vec{\alpha}\in A^{n-k}\rangle$$
in which $k$ ranges from $1$ to $n$ and $\mathscr{T}_n^{n}$ trivializes $\Phi$. These families are inductively defined alongside a series of related families $\mathscr{C}_k^n$, as described in items (3) and (4) below.
\item If $k$ is less than $n$ then $\mathscr{T}_k^{n}$ induces a family of $(n-k)$-coherent families of functions
$$\mathscr{C}_k^{n}=\langle\varsigma_{\vec{\alpha}}^{\vec{f}}\mid \vec{f}\in X^k\text{ and }\vec{\alpha}\in A^{n-k}\rangle .$$
To be precise, $\mathscr{C}_k^{n}$ is the union of the $(n-k)$-coherent families of functions
$$C_{n-k}^{\vec{f}}=\langle\varsigma_{\vec{\alpha}}^{\vec{f}}\mid \vec{\alpha}\in A^{n-k}\rangle$$
as $\vec{f}$ ranges through $X^k$.
\item Our inductive hypothesis ensures us trivializations $T_{n-k-1}^{\vec{f}}$ of each $C_{n-k}^{\vec{f}}$. These serve then to define
$$\mathscr{T}_{k+1}^n:=\bigcup_{\vec{f}\in X^k} T_{n-k-1}^{\vec{f}},$$and repeated, alternating applications of this and the previous step cumulatively yield the sequence $$\langle \mathscr{T}_k^n\mid 1\leq k\leq n\rangle$$ of item (2), as desired.
\end{enumerate}
Two points in the above scheme merit further discussion:
\begin{enumerate}
\item[(i)] We must specify precisely how the $(n-k)$-coherent families of functions $C_{n-k}^{\vec{f}}$ derive from the families $\mathscr{T}_k^{n}$, and verify that they are in fact $(n-k)$-coherent.
\item[(ii)] We must verify that $\mathscr{T}_n^{n}$ does indeed trivialize $\Phi$.
\end{enumerate}
We begin with item (i). The families $\mathscr{C}^n_k$ are inductively defined on positive integers $k\leq n$. The pattern when $k=1$ is plain enough from the examples above: for each $f\in X$ the subclass $C^f_{n-1}=\langle\varsigma^f_{\vec{\alpha}}\mid \vec{\alpha}\in A^{n-1}\rangle$ of $\mathscr{C}^n_1$ is defined from $\mathscr{T}_1^n=\langle\tau_{\vec{\alpha}}\mid \vec{\alpha}\in A^{n-1}\rangle$ by \begin{align}\label{defCn1} \varsigma_{\vec{\alpha}}^f=\varphi_{\vec{\alpha}f}+(-1)^n\tau_{\vec{\alpha}}\end{align}
for each $\vec{\alpha}\in A^{n-1}$. Observe then that for all $\vec{\alpha}\in A^n$,
\begin{align*}
\sum_{i=0}^{n-1}(-1)^i\varsigma_{\vec{\alpha}^i}^f & = \;\,\sum_{i=0}^{n-1}(-1)^i\varphi_{\vec{\alpha}^{i}\!f}+ (-1)^n\sum_{i=0}^{n-1}(-1)^i\tau_{\vec{\alpha}^i} \\ & =^{*} \sum_{i=0}^{n-1}(-1)^i\varphi_{\vec{\alpha}^{i}\!f} + (-1)^n\varphi_{\vec{\alpha}} \\ & =^{*} 0 ,
\end{align*}
by the coherence of $\Phi$. This shows that $C^f_{n-1}$ is $(n-1)$-coherent below $f$ and
hence, by the inductive hypothesis, $(n-1)$-trivial.

For the more general inductive definition of $\mathscr{C}_k^n$, suppose that the family $\mathscr{T}_k^n$ is defined; suppose also that the families $\mathscr{C}_j^n$ and $\mathscr{T}_j^n$ are defined for all $j<k$ and that each exhibits the coherence and trivialization features, respectively, described above. We then define $\mathscr{C}_k^n$ by letting \begin{align}\label{inductivedefinitionofCkn}\varsigma_{\vec{\alpha}}^{\vec{f}}=\varphi_{\vec{\alpha}\vec{f}}+(-1)^{n-k+1}\sum_{i=0}^{k-1}(-1)^i\tau_{\vec{\alpha}}^{\vec{f}^i}\end{align} for each $\vec{f}\in X^k$ and $\vec{\alpha}\in A^{n-k}$. Observe that equation (\ref{defCn1}) identifies naturally with the case of $k=1$.
\begin{claim} For each $\vec{f}\in X^k$ the family $C_{n-k}^{\vec{f}}=\langle\varsigma_{\vec{\alpha}}^{\vec{f}}\mid \vec{\alpha}\in A^{n-k}\rangle$ is $(n-k)$-coherent.
\end{claim}
\begin{proof} The more formal statement of the claim is that for each $\vec{f}\in X^k$ and $\vec{\alpha}\in A^{n-k+1}$,
\begin{align*}
\sum_{i=0}^{n-k}(-1)^i\varsigma_{\vec{\alpha}^i}^{\vec{f}} =^{*} 0\, .
\end{align*}
This is computationally verified as follows:
\begin{align*}
\sum_{i=0}^{n-k}(-1)^i\varsigma_{\vec{\alpha}^i}^{\vec{f}} & = \;\,\sum_{i=0}^{n-k}(-1)^i\varphi_{\vec{\alpha}^i\vec{f}}+(-1)^{n-k+1}\sum_{i=0}^{n-k}(-1)^i\sum_{j=0}^{k-1}(-1)^j\tau_{\vec{\alpha}^i}^{\vec{f}^j} \\ & =^{*}\sum_{i=0}^{n-k}(-1)^i\varphi_{\vec{\alpha}^i\vec{f}}+(-1)^{n-k+1}\sum_{j=0}^{k-1}(-1)^j\varsigma_{\vec{\alpha}}^{\vec{f}^j} \\ & =^{*}\sum_{i=0}^{n-k}(-1)^i\varphi_{\vec{\alpha}^i\vec{f}}+(-1)^{n-k+1}\bigg(\sum_{j=0}^{k-1}(-1)^j\Big(\varphi_{\vec{\alpha}\vec{f}^j}+(-1)^{n-k}\sum_{\ell=0}^{k-2}(-1)^{\ell}\tau_{\vec{\alpha}}^{(\vec{f}^j)^\ell}\Big)\bigg) \\ & =^{*} \sum_{i=0}^{n}(-1)^i\varphi_{(\vec{\alpha}\vec{f})^i}+(-1)^{n-k+1}\bigg(\sum_{j=0}^{k-1}(-1)^{j+n-k}\sum_{\ell=0}^{k-2}(-1)^{\ell}\tau_{\vec{\alpha}}^{(\vec{f}^j)^\ell}\bigg) \\ & =^{*} -\bigg(\sum_{j=0}^{k-1}(-1)^{j}\sum_{\ell=0}^{k-2}(-1)^{\ell}\tau_{\vec{\alpha}}^{(\vec{f}^j)^\ell}\bigg)\\ & =^{*} -\bigg(\sum_{j\leq \ell\leq k-2}(-1)^{j+\ell}\,\tau_{\vec{\alpha}}^{(\vec{f}^j)^\ell}+\sum_{\ell< j\leq k-1}(-1)^{j+\ell}\,\tau_{\vec{\alpha}}^{(\vec{f}^j)^\ell}\bigg) \\
& =^{*} -\bigg(\sum_{j\leq \ell\leq k-2}(-1)^{j+\ell}\,\tau_{\vec{\alpha}}^{(\vec{f}^{\ell+1})^j}+\sum_{\ell<j\leq k-1}(-1)^{j+\ell}\,\tau_{\vec{\alpha}}^{(\vec{f}^j)^\ell}\bigg) \\
& =^{*} -\bigg(\sum_{\ell<j\leq k-1}(-1)^{j+\ell+1}\,\tau_{\vec{\alpha}}^{(\vec{f}^{j})^\ell}+\sum_{\ell<j\leq k-1}(-1)^{j+\ell}\,\tau_{\vec{\alpha}}^{(\vec{f}^j)^\ell}\bigg)\\
& = 0\, .
\end{align*}
The fact that the functions $\tau_{\vec{\alpha}^i}^{\vec{f}^j}$ trivialize the functions $\varsigma_{\vec{\alpha}}^{\vec{f}^j}$ underlies the passage from the first line to the second; replace $\tau_{\vec{\alpha}}^{\vec{f}^j}$ with its definition at (\ref{inductivedefinitionofCkn}) to pass from the second line to the third. Nothing more than a regrouping underlies the passage from the third line to the fourth, whereupon the first sum vanishes by the $n$-coherence of $\Phi$. Simple bookkeeping converts the fifth line into the sixth, and the fact that $(\vec{f}^j)^\ell=(\vec{f}^{\ell+1})^j$ for all $j\leq\ell\leq k-2$ converts the sixth line into the seventh. A renaming of variables in the first sum then yields the eighth line, whose terms all plainly cancel.
\end{proof}
For future reference, we summarize this section's argument in the following lemma.
\begin{lemma}\label{summarypropagatinglemma}
Fix $n>1$ and cardinals $\kappa\leq\chi$ and let $\mathbb{P}=\mathrm{Add}(\omega,\chi)$. The following then holds in $V^{\mathbb{P}}$: suppose that
\begin{itemize}
\item $\Phi=\langle\varphi_{\vec{f}}\mid\vec{f}\in X^n\rangle$ is an $n$-coherent family of functions;
\item $A\subseteq X\subseteq\,^\omega\omega$ contains at least $\kappa$-many of the Cohen reals added by $\mathbb{P}$;
\item $\Phi\restriction A$ is trivial;
\item Any $j$-coherent family $\Psi=\langle\psi_{\vec{f}}\mid\vec{f}\in Y^j\rangle$ in which $1\leq j<n$ and $Y$ contains at least $\kappa$-many of the Cohen reals added by $\mathbb{P}$ is trivial.
\end{itemize}
Then $\Phi$ is trivial as well.
\end{lemma}

Note that, by Proposition \ref{below_prop}, the above lemma applies also to
$n$-coherent families below any fixed $g \in {^\omega}\omega$.

\section{Defining trivializations} \label{defining_trivializations_section}
To apply Lemma \ref{summarypropagatinglemma}, we must first show in the appropriate models $V^{\bb{P}}$ that $n$-coherent families $\Phi$ indexed by large numbers of Cohen reals always admit trivial restrictions $\Phi\restriction A$ to index-sets $A$ which are large in the settings of lower dimensions. This we argue by defining type II (i.e., finitely supported) trivializations of $\Phi\restriction A$. These definitions require variations on the machinery of \cite{SVHDL}; describing this machinery is the object of this section. Already in \cite{SVHDL}, however, this apparatus takes on a certain opacity; as this is, if anything, even more the case for the variations listed here, we conclude this section with several heuristic remarks.

\begin{definition}
  Suppose that $b$ is a finite set of ordinals.
  A \emph{subset-final segment} of $b$ of \emph{length $m$} is a
  sequence $\vec{a} = \langle a_i \mid 1 \leq i \leq m \rangle$ such that
  \begin{itemize}
    \item $m \leq |b|$,
    \item $a_1 \subseteq \cdots \subseteq a_m = b$, and
    \item $|a_i| = |a_1|+i-1$ for all $i$ with $1 \leq i \leq m$.
  \end{itemize}
  If $\vec{a}$ is a subset-final segment of $b$ and $|a_1| = 1$, then we
  say that $\vec{a}$ is a \emph{long string} or a \emph{long string for $b$}.
  Notice that in this case $m = |b|$. If $\vec{a}$ is not long, then it is \emph{short}.
\end{definition}

Suppose now that $X$ is a set of ordinals and we are working with an injective
sequence $\langle f_\alpha \mid \alpha \in X \rangle$ of elements of ${^\omega}\omega$.
(In the present context, this will always be a sequence of Cohen reals, but that is not
important for the results in this section.)
For each nonempty $\vec{\alpha} = \langle \alpha_k \mid k < n \rangle$ in $X^{<\omega}$,
let $I(\vec{\alpha})$ denote $\bigcap_{k < n} I(f_{\alpha_k})$.
Suppose that for each positive integer $n$ the family $\Phi_n =
\langle \varphi_{\vec{\alpha}} : I(\vec{\alpha}) \rightarrow \bb{Z} \mid
\vec{\alpha} \in X^n \rangle$ is $n$-coherent, and let $\vec{\Phi}$ denote the
family $\langle \varphi_{\vec{\alpha}} \mid \vec{\alpha} \in X^{<\omega} \rangle$.
Suppose also that to each nonempty $a \in [X]^{<\omega}$ we have assigned an ordinal
$\varepsilon_a \in X$ in such a way that
\begin{itemize}
  \item if $a = \{\alpha\}$, then $\varepsilon_a = \alpha$;
  \item if $a \subsetneq b$, then $\varepsilon_a < \varepsilon_b$.
\end{itemize}
Now, given a nonempty $b \in [X]^{<\omega}$ and a subset-final segment
$\vec{a} = \langle a_k \mid 1 \leq k \leq m \rangle$ of $b$, define the set
$d^{\vec{\varepsilon}}_{\vec{a}}$ as follows. If $|a_1| = 1$, then let
$d^{\vec{\varepsilon}}_{\vec{a}} = \{\varepsilon_{a_k} \mid 1 \leq k \leq m\}$.
Note that, in this case, $d^{\vec{\varepsilon}}_{\vec{a}} \in [X]^{|b|}$.
If $|a_1| > 1$, then let $d^{\vec{\varepsilon}}_{\vec{a}} =
a_1 \cup \{\varepsilon_{a_k} \mid 1 \leq k \leq m\}$. Note that, in this case,
$d^{\vec{\varepsilon}}_{\vec{a}} \in [X]^{|b|+1}$.

For $\vec{\alpha} \in X^{<\omega}$ of length at least two, let
\[
  e^{\vec{\Phi}}(\vec{\alpha}) = \sum_{i < |\vec{\alpha}|}(-1)^i \varphi_{\vec{\alpha}^i}.
\]
When the family $\vec{\Phi}$ is clear from context, we will omit it from
the superscript; similarly for the superscript of $d^{\vec{\varepsilon}}_{\vec{a}}$. Recall also our habit of viewing finite sets $a \in [X]^{<\omega}$
as sequences enumerated in increasing order; expressions like $e(a)$ should
be interpreted on this principle. Since each $\Phi_n$ is $n$-coherent, $e(\vec{\alpha})$ is finitely
supported for each $\vec{\alpha} \in X^{<\omega}$. Let $\mathsf{e}(\vec{\alpha})$
denote the restriction of $e(\vec{\alpha})$ to its support.

We will be interested in linear combinations $L$ of the form
\[
  \sum_{i < \ell} c_i e(\vec{\alpha}_i),
\]
where $\ell < \omega$, each $c_i$ is an integer, and each $\vec{\alpha}_i$ is
an element of $X^{<\omega}$ of length at least two.\footnote{In settings like these, notational choices are simply of the lesser evil. We will write $\vec{\alpha}_i(j)$ for the $j^{\mathrm{th}}$ element of $\vec{\alpha}_i$.} Given such a linear combination
$L$ and an ordinal $\varepsilon \in X$, we let the expression $L \ast \varepsilon$
denote
\[
  \sum_{i < \ell} c_i e(\vec{\alpha}_i {}^\frown \langle \varepsilon \rangle).
\]

For integers $n \geq 2$, we now define interrelated
\begin{itemize}
  \item linear combinations $\mathcal{A}^{\vec{\Phi}}_n(a)$, parametrized by
  $a \in [X]^n$, and
  \item linear combinations $\mathcal{C}^{\vec{\Phi}}_n(b)$, parametrized by
  $b \in [X]^{n+1}$.
\end{itemize}
We again omit the superscripts $\vec{\Phi}$ and $\vec{\varepsilon}$ and restriction-notations whenever they are contextually clear; as elsewhere, sums of functions in expressions like $\mathcal{C}_n(b)$ below should always be understood to be taken on the intersection of those functions' domains.

We begin our definitions by letting
\[
  \mathcal{A}_2(a) = e(a ^\frown \langle \varepsilon_a \rangle).
\]
for each $a \in [X]^2$. Next, suppose that $2 \leq n < \omega$ and $\mathcal{A}_n(a)$ has been defined for all
$a \in [X]^n$. Given $b \in [X]^{n+1}$, let
\begin{align}
  \mathcal{C}_n(b) &= e(b) - \sum_{i=0}^n (-1)^i \mathcal{A}_n(b^i), \text{ and } \\
  \mathcal{A}_{n+1}(b) &= (-1)^{n+1}\mathcal{C}_n(b) * \varepsilon_b. \label{C_step}
\end{align}

The following lemma is easily verified by induction on $n$, so its proof is left
to the reader.

\begin{lemma} \label{a_c_form_lemma}
  For all $b \in [X]^2$, letting $\vec{a} = \langle b \rangle$, we have $\mathcal{A}_2(b)
  = e(d_{\vec{a}})$.

  For all $n$ with $2 \leq n < \omega$ and all $b \in [X]^{n+1}$, we have:
  \begin{enumerate}
    \item $\mathcal{C}_n(b)$ is of the form
    \[
      e(b) + \sum_{i < \ell}c_i e(d_{\vec{a}_i}),
    \]
    where $\ell < \omega$ and, for each $i < \ell$, $c_i$ is an integer and $\vec{a}_i$
    is a short subset-final segment of some element of $[b]^n$.
    \item $\mathcal{A}_{n+1}(b)$ is of the form
    \[
      \sum_{i < \ell} c_i e(d_{\vec{a}_i}),
    \]
    where $\ell < \omega$ and, for each $i < \ell$, $c_i$ is an integer and
    $\vec{a}_i$ is a short subset-final segment of $b$.
  \end{enumerate}
\end{lemma}

One consequence of this lemma is that, while the expressions defining $\mathcal{A}_n(a)$
and $\mathcal{C}_n(b)$ are constructed via a recursion involving $\vec{\Phi}$,
the actual values of $\mathcal{A}_n$ and $\mathcal{C}_n$ are only dependent on
$\Phi_n$, so we can meaningfully speak of them in situations in which we have
only $\Phi_n$, and not $\Phi_{m}$ for any $m \neq n$, before us. In addition,
the values of $\mathcal{A}_n(a)$ and $\mathcal{C}_n(b)$ are only dependent on
ordinals $\varepsilon_a$ for nonempty $c \subseteq a$ or $c \subseteq b$,
respectively, so, when working just with the expressions
$\mathcal{A}_n(a)$ or $\mathcal{C}_n(b)$, we need not require that $\varepsilon_c$
is defined for any $c$ that is not a subset of $a$ or $b$, respectively. Finally,
if $b \in [X]^{n+1}$ and $(j,k) \in \omega \times \omega$ is an element of
$I(f_{\varepsilon_a})$ for all nonempty
$a \subseteq b$, then $(j,k)$ is in the domain of $\mathcal{C}_n(b)$.

The following is a consequence of \cite[Lemma 6.4]{SVHDL}.

\begin{lemma} \label{cancellation_fact}
  Suppose that $2 \leq n < \omega$, $b \in [X]^{n+1}$, $(j,k) \in \omega \times
  \omega$, and the following two statements hold.
  \begin{itemize}
    \item There exists a single integer $w$ such that
    $\mathsf{e}(d_{\vec{a}})(j,k)
    = w$ for every long string $\vec{a}$ for $b$.
    \item $(j,k) \in I(f_{\varepsilon_a})$ for all nonempty $a \subseteq b$.
  \end{itemize}
  Then $\mathcal{C}_n(b)(j,k) = 0$.
\end{lemma}

\begin{remark}
  We briefly describe how Lemma \ref{cancellation_fact} follows
  from the argument of \cite[Lemma 6.4]{SVHDL}. The set $b$ here corresponds to
  $\tau$ in that result. The expression $\mathcal{S}_n(\tau)$ in \cite{SVHDL}
  is an auxiliary expression that always equals $0$. The two assumptions
  in Lemma \ref{cancellation_fact} play the role of the statement $\mathfrak{u}_n(\tau)$
  from \cite{SVHDL}.
\end{remark}

As we will see, the significance of these definitions is the following: the expressions $\mathcal{A}_n(a)$ will correspond to the elements of candidate type II trivializations of $\Phi\restriction A$. Under the conditions of Lemma \ref{cancellation_fact}, the expressions $\mathcal{C}_n(b)$ amount to verifications that these families of expressions $\mathcal{A}_n(a)$ do indeed trivialize $\Phi\restriction A$. These relations, together with the necessity of working coordinatewise, as in Lemma \ref{cancellation_fact}, are points we expand on in the following remark. Though not strictly needed for the continuation of our argument, it is hoped that it may be clarifying.

\begin{remark} The variability in $|d_{\vec{a}}|$ noted above, depending on whether $|a_1|=1$, underscores the unique status of the sets $d_{\vec{a}}$ which are indexed by long strings; it is on the cancellations between their associated terms $\mathsf{e}(d_{\vec{a}})$ that the desired relations between the other terms in $\mathcal{C}_n(b)$ depend. This we hope to illuminate by the following diagram and discussion.

  \begin{figure}[H]
  \centering
  \begin{tikzpicture}
  \fill[fill=gray, opacity=.2] (0,0) -- (6,0) -- (3,.7) -- cycle;
  \fill[fill=gray, opacity=.2] (0,0) -- (3,5.196) -- (2.106,2.248) -- cycle;
  \fill[fill=gray, opacity=.2] (6,0) -- (3,5.196) -- (3.752,2.248) -- cycle;
  \draw (0,0) node[circle, fill=white, inner sep=1pt]{$0$}
    -- (6,0) node[circle, fill=white, inner sep=1pt]{$1$}
      -- (3,.7) node[circle, fill=white, inner sep=1pt]{$01$}
    -- cycle;
    \draw (0,0) node[circle, fill=white, inner sep=1pt]{$0$}
    -- (3,5.196) node[circle, fill=white, inner sep=1pt]{$2$}
      -- (2.106,2.248) node[circle, fill=white, inner sep=1pt]{$02$}
    -- cycle;
    \draw (6,0) node[circle, fill=white, inner sep=1pt]{$1$}
    -- (3,5.196) node[circle, fill=white, inner sep=1pt]{$2$}
      -- (3.752,2.248) node[circle, fill=white, inner sep=1pt]{$12$}
    -- cycle;
  \draw (3,5.196) node[circle, fill=white, inner sep=1pt]{$2$}
  -- (3,1.732) node[circle, fill=white, inner sep=1pt]{$012$};
  \draw (6,0) node[circle, fill=white, inner sep=1pt]{$1$}
  -- (3,1.732) node[circle, fill=white, inner sep=1pt]{$012$};
  \draw (0,0) node[circle, fill=white, inner sep=1pt]{$0$}
  -- (3,1.732) node[circle, fill=white, inner sep=1pt]{$012$};
  \draw (2.106,2.248) node[circle, fill=white, inner sep=1pt]{$02$}
  -- (3,1.732) node[circle, fill=white, inner sep=1pt]{$012$};
  \draw (3.752,2.248) node[circle, fill=white, inner sep=1pt]{$12$}
  -- (3,1.732) node[circle, fill=white, inner sep=1pt]{$012$};
  \draw (3,.7) node[circle, fill=white, inner sep=1pt]{$01$}
  -- (3,1.732) node[circle, fill=white, inner sep=1pt]{$012$};
  \draw (2.2,1.7) node {$w$};
  \draw (3.8,1.7) node {$w$};
  \draw (2.55,1.1) node {$w$};
  \draw (3.45,1.1) node {$w$};
  \draw (2.65,2.4) node {$w$};
  \draw (3.32,2.4) node {$w$};
  \end{tikzpicture}
  \caption{The subdivision organizing the $n=2$ case.}
  \label{subdivision_figure}
  \end{figure}
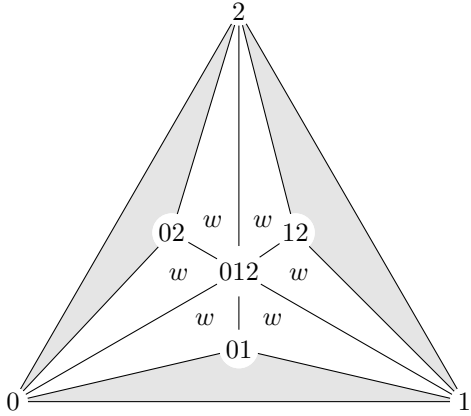

The $n$-tuples $\vec{f}$ structuring our various coherent families $\Phi$ are naturally viewed as simplices, with $(n-1)$-tuples $\vec{f}^i$ $(i<n)$ as faces. Within this view, strings of increasing subsets $\dots\subsetneq a_i\subsetneq\cdots$ of $b=\vec{f}$ correspond to simplices in the \emph{barycentric subdivision} of the simplex $b$ (see, e.g., \cite[Chapter 3.3]{spanier}). For example, the white inner region of Figure \ref{subdivision_figure} corresponds to the barycentric subdivision of the triangle with vertices $0$, $1$, and $2$ (we bend that region's edges for reasons soon to be made clear): the $2$-faces of that subdivided region are $\{0,01,012\}$, $\{1,01,012\}$, etc., each corresponding to inclusion-increasing sequences of nonempty subsets of $\{0,1,2\}$.

Fix now a $2$-coherent family of functions $\Phi$. Let $a =\{\alpha_0,
\alpha_1, \alpha_2\} \in [X]^3$ and suppose that the hypotheses of Lemma
\ref{cancellation_fact} hold. For readability, we let, for example,
$\varepsilon_{0,1}$ denote $\varepsilon_{\alpha_0, \alpha_1}$ (in particular,
$\varepsilon_{\ell} = \alpha_\ell$ for $\ell < 3$). Writing $\varphi_{0,1}$ for
the function indexed by $(f_{\varepsilon_0},f_{\varepsilon_1})$, and so on,
what are wanted are \emph{finitely supported} functions $\psi_{0,1}$, and so on, whose differences reproduce those among the corresponding functions of $\Phi$, as described in Proposition \ref{trivial_equivalence_fact}. The idea of the above machinery is to derive these finitely supported functions from the coherence of $\Phi$ itself, as the differences between carefully chosen families of functions $\varphi$; at the stage $n=2$, for example, we will have
\begin{align}\label{schematic_A2}\psi_{0,1}:=\mathcal{A}_2(\varepsilon_0,\varepsilon_1)=\varphi_{1,01}-\varphi_{0,01}+\varphi_{0,1}.\end{align}
  Visually, this definition corresponds to the grey triangle at the base of Figure \ref{subdivision_figure}, under the natural association of the functions $\varphi_{1,01}$, $\varphi_{0,01}$, and $\varphi_{0,1}$ with the edges $\{1,01\}$, $\{0,01\}$, and $\{0,1\}$, respectively. Under this correspondence, the desired relation
\begin{align}\label{schematic_equality}\varphi_{1,2}-\varphi_{0,2}+\varphi_{0,1}=\psi_{1,2}-\psi_{0,2}+\psi_{0,1}\end{align}
may be viewed as asserting the equality of the oriented sum of the functions $\varphi$ associated to the boundary of the triangle $\{0,1,2\}$ with that of the functions $\varphi$ associated to the boundary of the grey region. (As should be clear, this is a deliberately schematic discussion; we return to the question of the argument of these functions below.) This holds precisely because of the first bulleted ``long string'' condition listed in Lemma \ref{cancellation_fact}, which amounts in the present context to the boundary sums associated to the triangles $\{0,01,012\}$, $\{1,01,012\}$, etc., all equaling $w$. As these are oriented sums, they entail cancellations, so that first, the boundary sum associated to $\{0,1,01\}$ may be identified with that associated to $\{0,1,012\}$, and second, such identifications for each of the grey triangles cancel inside the triangle $\{0,1,2\}$, leaving nothing summed but its boundary, just as equation \ref{schematic_equality} requires.

At the arithmetic level, all of this manifests (with only minor notational adjustments) as exactly the two types of cancellations in the summed equations concluding Section 6 of \cite{SVHDL}. The ensuing simplification of that sum is an instance of what Lemma \ref{cancellation_fact} records as $\mathcal{C}_n(b)=0$, which translates, in turn, to the equation \ref{schematic_equality} we had desired.

The subdivision perspective sketched above is valuable for returning sense to what appear here or in \cite{SVHDL} as rather opaque and complicated algebraic identities: the meaning of those identities is that, in forcing extensions, higher-dimensional $\Delta$-systems can determine trivializing structures within $n$-coherent families by uniformizing the boundary sums associated to the $n$-faces of the barycentric subdivision of any $(n+1)$-tuple of indices, viewed as a simplex. This perspective clarifies the passage from one dimension to the next, as well; as the interested reader may verify, in the $n=3$ case, along with with the face $\{0,1,2\}$, the $2$-faces pictured in Figure \ref{subdivision_figure} play within a tetrahedron exactly the role that the boundaries of the grey faces had played within a triangle in the case of $n=2$. Put differently, the way verification-expressions $\mathcal{C}_n$ figure in the trivializing expressions $\mathcal{A}_{n+1}$ of the next level, as in equation \ref{C_step}, amounts to little other than the fact that the restriction of the barycentric subdivision of an $(n+1)$-simplex to any $n$-face is a barycentric subdivision of that face.

Complicating the above considerations, however, is the issue of domain: for the right-hand side of equations like (\ref{schematic_equality}) to truly be trivializing in the sense of Proposition \ref{trivial_equivalence_fact}, the domain of $\psi_{0,1}$, like that of $\varphi_{0,1}$, must be $I(f_{\varepsilon_0}\wedge f_{\varepsilon_1})$, and similarly for the functions $\psi_{0,2}$ and $\psi_{1,2}$. If $\psi_{0,1}$ is defined as in equation \ref{schematic_A2}, then this amounts to a requirement that $f_{\varepsilon_{01}}\geq f_{\varepsilon_0}\wedge f_{\varepsilon_1}$, which, if the functions indexed are Cohen reals, can never be the case. This is a requirement we can only meet locally, choosing for each $(j,k)\in I(f_{\varepsilon_0}\wedge f_{\varepsilon_1})$ an $\varepsilon_{01}^{j,k}$ such that $(j,k)\in I(f_{\varepsilon_{01}^{j,k}})$. This is the approach we take, and this is the meaning of the parameter $(j,k)$ appearing in Lemma \ref{cancellation_fact}. The good news in this approach is that equations like (\ref{schematic_equality}) hold if and only if they hold coordinatewise, so that our arithmetic is essentially unaffected. The bad news is that functions like $\psi_{0,1}$ may now fail to be finitely supported, and much of the work of the following section is towards ensuring that they will be.
\end{remark}
\section{The cases of $n>1$}\label{highernsection}
We turn now to the remainder of the proof of our main theorem. We will prove the more precise statement given immediately below. The cardinals $\lambda_n$ appearing therein are defined by recursion on $n\geq 1$ as follows. First, as in Theorem \ref{1d_theorem}, let $\lambda_1 = \beth_1^+$; then, for all $n >1$, let $\lambda_n =
\sigma(\lambda^+_{n-1}, 2n+1)$ (again see Definition \ref{sigmadefinition} for the notation $\sigma(\,\cdot\,,\,\cdot\,)$). Note that $\sup\{\lambda_n \mid n < \omega\} = \beth_\omega$.
It is also readily verified that each $\lambda_n$ is
${<}\aleph_1$-inaccessible; in consequence, since $\lambda_n$ is a successor cardinal,
$\lambda_n^+$ is also ${<}\aleph_1$-inaccessible.

\begin{theorem}\label{maintheorem}
  Let $n$ be a positive integer, let $\chi \geq \lambda_n$ be a cardinal, and let
  $\bb{P} = \mathrm{Add}(\omega, \chi)$. The following then holds in $V^{\bb{P}}$:
  For any set $X \subseteq {^\omega} \omega$ containing at least $\lambda_n$-many
  of the Cohen reals added by $\bb{P}$, every $n$-coherent family
  $\Phi = \langle \varphi_{\vec{f}} \mid \vec{f} \in X^n\rangle$ indexed by $X$
  is trivial.
\end{theorem}

\begin{proof} The proof is by induction on $n$. The case
  $n=1$ was that of Theorem \ref{1d_theorem}. Therefore fix an $n >1$ and suppose the theorem proven for all
  positive $m < n$. Also fix a cardinal $\chi \geq \lambda_n$, $\bb{P}$-names $\dot{X}$ and
  $\dot{\Phi} = \langle \dot{\varphi}_{\dot{\vec{f}}} \mid \dot{\vec{f}} \in
  \dot{X}^n \rangle$, and a condition $p\in \bb{P}$ such that
  \begin{itemize}
    \item $p \Vdash ``|\{\alpha < \chi \mid \dot{f}_\alpha \in \dot{X}\}| \geq
    \lambda_n"$, and
    \item $p \Vdash ``\dot{\Phi} \text{ is an }n\text{-coherent family}"$.
  \end{itemize}

  We will find a $q \leq p$ and a $\bb{P}$-name $\dot{A}$ such that $q$ forces the
    following statements:
    \begin{itemize}
      \item $|\dot{A}| \geq \lambda_{n-1}$;
      \item $\{\dot{f}_\alpha \mid \alpha \in \dot{A}\} \subseteq \dot{X}$;
      \item $\dot{\Phi} \restriction \dot{A}$ is trivial.
    \end{itemize}
    It will be clear from our argument below that each induction step of our proof conserves the hypotheses of Lemma \ref{summarypropagatinglemma}. That
    lemma will therefore apply to show that $q$ in fact forces that $\dot{\Phi}$ is
    trivial; this will conclude the induction step of the proof and, therefore,
    the proof itself.

    As before, begin by letting $Y$ be the set of $\alpha < \chi$ for which there
    is a condition $p_\alpha \leq p$ such that $p_\alpha \Vdash ``\dot{f}_\alpha
    \in \dot{X}"$; observe that $|Y| \geq \lambda_n$ by assumption. For each
    $\alpha \in Y$, fix such a condition $p_\alpha$. Since $\bb{P}$ is
    $\lambda_n$-Knaster, there exists a set $Y' \subseteq Y$ of size $\lambda_n$
    such that $\{p_\alpha \mid \alpha \in Y'\}$ consists of pairwise compatible
    conditions. Note that, for all $a \in [Y']^{<\omega}$, we have
    $\bigcup_{\alpha \in a} p_\alpha \in \bb{P}$.

    Given a $\bb{P}$-name $\dot{\vec{h}} = \langle \dot{h}_0, \ldots, \dot{h}_n \rangle$
    for an element of $({^\omega}\omega)^{n+1}$, let $\dot{e}(\dot{\vec{h}})$
    be a $\bb{P}$-name that is forced to be equal to
    \[
    \sum_{i = 0}^n (-1)^i \varphi_{\vec{h}^i}
    \] if $\dot{\vec{h}} \in \dot{X}^{n+1}$ and is forced
    to be $0$ otherwise. Since $p$ forces that $\Phi$ is $n$-coherent, any extension
    of $p$ will force that $\dot{e}(\dot{\vec{h}})$ is a finitely-supported function
    from a subset of $\omega \times \omega$ into $\bb{Z}$. Let $\dot{\mathtt{e}}(\dot{\vec{h}})$
    be a $\bb{P}$-name for the restriction of $\dot{e}(\dot{\vec{h}})$ to its support.
    For all $a \in [Y']^{n+1}$, let $\dot{e}(a)$ denote $\dot{e}(\langle
    \dot{f}_\alpha \mid \alpha \in a \rangle)$.

    For each $a \in [Y']^{n+1}$ let $\langle q_{a, \ell} \mid \ell < \omega \rangle$
    enumerate a maximal antichain $\mathcal{A}_a$ of conditions in $\bb{P}$ below
    $\bigcup_{\alpha \in a} p_\alpha$ such that each $q_{a, \ell}$ decides the value of $\dot{\mathsf{e}}(a)$ to be equal to some
    finite partial function $\mathsf{e}_{a, \ell} \in V$. Recall that, for $p \in \bb{P}$,
    $u(p)$ is the set $\{\alpha < \chi \mid \dom(p) \cap (\{\alpha\} \times \omega) \neq
    \emptyset\}$. For readability, let $u(a,\ell)$ denote $u(q_{a,\ell})$.

    For each $b \in [Y']^{2n+1}$ let $v_b = \bigcup \{u(a,\ell) \mid a \in [b]^{n+1},
    ~ \ell < \omega \}$. Define a ``coding'' function $F:[Y']^{2n+1} \rightarrow H(\omega_1)$ as
    follows. First, for each $b \in [Y']^{2n+1}$, each $\mb{m} \in [2n+1]^{n+1}$ and each $\ell < \omega$,
    let
    \[
      w^b_{\mb{m}, \ell} = \{\eta < \otp(v_b) \mid v_b(\eta) \in u(b[\mb{m}], \ell)\}.
    \]
    (Note that $w^b_{\mb{m}, \ell} \in [\otp(v_b)]^{<\omega}$). Then, for each $b \in [Y']^{2n+1}$, let
    \[
      F(b) = \langle (\bar{q}_{b[\mb{m}], \ell}, w^b_{\mb{m}, \ell},
      \mathsf{e}_{b[\mb{m}], \ell}) \mid \mb{m} \in [2n+1]^{n+1}, ~ \ell < \omega \rangle.
    \]
    Recall that $\lambda_n = \sigma(\lambda^+_{n-1}, 2n+1)$ and $\lambda^+_{n-1}$ is
    ${<}\aleph_1$-inaccessible. Therefore, by Fact \ref{delta_systems_fact}, there exists $H \in [Y']^{\lambda^+_{n-1}}$ such that
    \begin{itemize}
      \item $F$ is constant on $[H]^{2n+1}$, taking value $\langle (\bar{q}_{\mb{m}, \ell},
      w_{\mb{m}, \ell}, \mathsf{e}_{\mb{m}, \ell}) \mid \mb{m} \in [2n+1]^{n+1}, \ell < \omega \rangle$, and
      \item $\langle v_b \mid b \in [H]^{2n+1} \rangle$ is a uniform $(2n+1)$-dimensional
      $\Delta$-system.
    \end{itemize}
    By taking an initial segment of $H$ if necessary, we can assume that
    $\otp(H) = \lambda^+_{n-1}$.
    Let $\rho$ and $\langle \mb{r}_{\mb{m}} \mid \mb{m} \subseteq 2n+1 \rangle$
    witness that $\langle v_b \mid b \in [H]^{2n+1} \rangle$ is a uniform $(2n+1)$-dimensional
    $\Delta$-system.

    Let $\langle v_a \mid a \in [H]^{<2n+1} \rangle$ be given
    by Lemma \ref{extension_lemma} applied to $\langle v_b \mid b \in [H]^{2n+1} \rangle$.
    We will actually need slightly more than what Lemma \ref{extension_lemma} gives us.
    Given $a \in [H]^n$, $k \leq n$, and $\alpha \in H$, we say that $\alpha$ is
    \emph{$k$-addable for $a$} if $\alpha \notin a$ and $|a \cap \alpha| = k$. In
    other words, $\alpha$ is $k$-addable to $a$ if, letting $a' = a \cup \{\alpha\}$,
    we have $|a'| = n+1$ and $a'(k) = \alpha$. Given an $a \in [H]^n$ and a $k \leq n$
    such that there is at least one $\alpha \in H$ that is $k$-addable for $a$, define
    $v_{a, k}$ as follows. Let $\alpha \in H$ be such that $\alpha$ is addable for
    $a$, let $b \in [H]^{2n+1}$ be such that $b[n+1] = a \cup \{\alpha\}$, and
    let $v_{a,k} = v_b[\mb{r}_{(n+1) \setminus \{k\}}]$.

    \begin{claim} \label{ds_claim_1}
      For each $a \in [H]^n$ and $k \leq n$ for which $v_{a,k}$ is defined, the
      value of $v_{a,k}$ is independent of our choice of $\alpha$ and $b$.
    \end{claim}

    \begin{proof}
      Suppose that $\alpha, \alpha' \in H$ are both $k$-addable for $a$ and
      $b,b' \in [H]^{2n+1}$ are such that $b[n+1] = a \cup \{\alpha\}$ and
      $b[n+1] = a \cup \{\alpha'\}$. We will show that $v_b[\mb{r}_{(n+1) \setminus \{k\}}]
      = v_{b'}[\mb{r}_{(n+1) \setminus \{k\}}]$.

      First, fix $c \in [H]^n$ such that $\min(c) > \max(b \cup b')$, let
      $d = a \cup \{\alpha\} \cup c$, and let $d' = a \cup \{\alpha'\} \cup c$.
      Then $b$ and $d$ are aligned, with $\mb{r}(b,d) = n+1$, and
      $b'$ and $d'$ are aligned, with $\mb{r}(b',d') = n+1$. Also,
      $d$ and $d'$ are aligned, with either $\mb{r}(d,d') = n+1$ (if $\alpha = \alpha'$)
      or $\mb{r}(d, d') = (n+1) \setminus \{k\}$ (if $\alpha \neq \alpha'$). Altogether,
      it follows that
      \[
        v_b[\mb{r}_{(n+1) \setminus \{k\}}] = v_d[\mb{r}_{(n+1) \setminus \{k\}}] =
        v_{d'}[\mb{r}_{(n+1) \setminus \{k\}}] = v_{b'}[\mb{r}_{(n+1) \setminus \{k\}}],
      \]
      as desired.
    \end{proof}

    \begin{claim} \label{ds_claim_2}
      Suppose that $a \in [H]^n$ and $k \leq n$ are such that $v_{a,k}$ is defined.
      Then the collection
      \[
        \{v_{a \cup \{\alpha\}} \mid \alpha \in H \text{ is addable for } a\}
      \]
      is a (1-dimensional) $\Delta$-system, with root $v_{a, k}$.
    \end{claim}

    \begin{proof}
      Fix $\alpha < \alpha'$ in $H$ such that both $\alpha$ and $\alpha'$ are
      $k$-addable for $a$. Fix $c \in [H]^n$ such that $\min(c) > \max(a \cup \{\alpha,
      \alpha'\})$. Let $b = a \cup \{\alpha\} \cup c$ and $b' = a \cup \{\alpha'\}
      \cup c$. Then $v_{a \cup \{\alpha\}} = v_b[\mb{r}_{n+1}]$,
      $v_{a \cup \{\alpha'\}} = v_{b'}[\mb{r}_{n+1}]$, and
      $v_{a,k} = v_b[(n+1) \setminus \{k\}] = v_{b'}[(n+1) \setminus \{k\}]$.
      Moreover, $b$ and $b'$ are
      aligned with $\mb{r}(b,b') = (2n+1) \setminus \{k\}$, so
      $v_b \cap v_{b'} = v_b[(2n+1) \setminus \{k\}] =
      v_{b'}[(2n+1) \setminus \{k\}]$. Altogether, this implies
      \begin{align*}
        v_{a \cup \{\alpha\}} \cap v_{a \cup \{\alpha'\}}
        &= v_b[\mb{r}_{n+1}] \cap v_{b'}[\mb{r}_{n+1}] \\
        &= v_b[\mb{r}_{n+1}] \cap v_{b'}[\mb{r}_{n+1}]
        \cap v_b[(2n+1) \setminus \{k\}] \cap v_{b'}[(2n+1) \setminus \{k\}] \\
        &= v_b[\mb{r}_{(n+1) \setminus \{k\}}] \cap v_{b'}[\mb{r}_{(n+1) \setminus \{k\}}] \\
        &= v_{a, k} \cap v_{a,k} = v_{a,k}.
      \end{align*}
      Therefore, $\{v_{a \cup \{\alpha\}} \mid \alpha \in H \text{ is addable for } a\}$
      is a $\Delta$-system with root $v_{a,k}$.
    \end{proof}

    \begin{claim} \label{ds_claim_3}
      Let $a \in [H]^n$, $k \leq n$, and $\ell < \omega$, and suppose that
      $\alpha, \alpha' \in H$ are both $k$-addable for $a$. Then
      $q_{a \cup \{\alpha\}} \restriction (v_{a,k} \times \omega) =
      q_{a \cup \{\alpha'\}} \restriction (v_{a,k} \times \omega)$.
    \end{claim}

    \begin{proof}
      We can assume that $\alpha \neq \alpha'$, as otherwise the claim is trivial.
      Fix $c \in [H]^n$ with $\min(c) > \max(a \cup \{\alpha, \alpha'\})$, and
      let $b = a \cup \{\alpha\} \cup c$ and $b' = a \cup \{\alpha'\} \cup c$.
      By Claim \ref{ds_claim_1}, we have $v_{a,k} = v_b[\mb{r}_{(n+1) \setminus \{k\}}]
      = v_{b'}[\mb{r}_{(n+1) \setminus \{k\}}]$.

      Now suppose that $(\gamma, j) \in \dom(q_{a \cup \{\alpha\}}) \cap
      (v_{a,k} \times \omega)$. Then there is $\eta \in \mb{r}_{(n+1) \setminus \{k\}}$ such that
      $\gamma = v_b(\eta)$; moreover, $\eta \in w_{n+1, \ell}$ and, since
      $b$ and $b'$ are aligned with $\mb{r}(b,b') \supseteq (n+1) \setminus \{k\}$,
      we have $\gamma = v_{b'}(\eta)$. Let $i < \omega$ be such that
      $\eta = w_{n+1, \ell}(i)$. Then, since $\bar{q}_{a \cup \{\alpha\}, \ell} =
      \bar{q}_{a \cup \{\alpha'\}, \ell} = \bar{q}_{n+1, \ell}$, we have $(\gamma, j) \in
      \dom(q_{a \cup \{\alpha'\}})$ and
      \[
        q_{a \cup \{\alpha'\}}(\gamma, j) = \bar{q}_{n+1, \ell}(i, j) =
        q_{a \cup \{\alpha\}}(\gamma, j).
      \]

      A symmetric argument shows that, if $(\gamma, j) \in \dom(q_{a \cup \{\alpha'\}})
      \cap (v_{a,k} \times \omega)$, then $(\gamma, j) \in \dom(q_{a \cup \{\alpha\}})$
      and $q_{a \cup \{\alpha\}}(\gamma, j) = q_{a \cup \{\alpha'\}}(\gamma, j)$.
      It follows that $q_{a \cup \{\alpha\}} \restriction (v_{a,k} \times \omega) =
      q_{a \cup \{\alpha'\}} \restriction (v_{a,k} \times \omega)$.
    \end{proof}

    We next note that the values of $\bar{q}_{\mb{m}, \ell}$ and $\mathsf{e}_{\mb{m}, \ell}$
    are independent of $\mb{m} \in [2n+1]^{n+1}$. Indeed, suppose that $b \in [H]^{2n+1}$,
    $\mb{m} \in [2n+1]^{n+1}$, and $a = b[\mb{m}]$. Then we can find a $b^* \in [H]^{2n+1}$
    for which $a = b^*[n+1]$ (i.e., $a$ is an initial segment of $b^*$). But then
    \[
      \bar{q}_{\mb{m}, \ell} = \bar{q}_{b[\mb{m}], \ell} = \bar{q}_{a, \ell} =
      \bar{q}_{b^*[n+1], \ell} = \bar{q}_{n+1, \ell},
    \]
    and similarly for $\mathsf{e}_{\mb{m}, \ell}$. Hence we may in fact fix a sequence
    $\langle (\bar{q}_\ell, \mathsf{e}_\ell) \mid \ell < \omega \rangle$ such that
    $(\bar{q}_{\mb{m}, \ell}, \mathsf{e}_{\mb{m}, \ell}) = (\bar{q}_\ell, \mathsf{e}_\ell)$
    for all $\mb{m} \in [2n]^{n+1}$.

   Now let $H_0, H_1, \ldots, H_n$ be subsets of $H$ such that  $\otp(H_k) = \lambda_{n-1}$ and $H_k < H_{k'}$ for all $k < k' \leq n$. For each $k \leq n$ let $\delta_k = \min(H_k)$, and let
    $d = \{\delta_k \mid k \leq n\}$. Let $q = q_{d,0}$; note that $q\leq p$. Let $d^+ = \{\delta_k \mid 1 \leq k \leq n\}$ and for all $\alpha \in H_0$
    let $d_\alpha = \{\alpha\} \cup d^+$ be the result of replacing $\delta_0$ by
    $\alpha$ in $d$. Let $\dot{A}$ be a $\bb{P}$-name for the set of
    $\alpha \in H_0$ such that $q_{d_\alpha, 0} \in \dot{G}$. We claim
    that $q$ and $\dot{A}$ are as desired; more precisely, we claim that $q$ forces that
    \begin{enumerate}
      \item[(i)] $|\dot{A}| \geq \lambda_{n-1}$,
      \item[(ii)] $\{\dot{f}_\alpha \mid \alpha \in \dot{A}\} \subseteq \dot{X}$, and
      \item[(iii)] $\dot{\Phi} \restriction \dot{A}$ is trivial.
    \end{enumerate}
    We can dispense with the first two of these items fairly quickly.

    \begin{claim}\label{Aunboundedclaim}
      $q \Vdash ``\dot{A} \text{ is unbounded in }H_0"$.
    \end{claim}

    \begin{proof}
      Fix an arbitrary $\gamma \in H_0$ and an arbitrary condition $r \leq q$. We will find
      an $\alpha \in H_0 \setminus \gamma$ such that $q_{d_\alpha, 0}$ is compatible
      with $r$. A routine density argument will then establish the claim.

      By Claim \ref{ds_claim_1}, $\{v_{d_\alpha} \mid \alpha \in H_0 \setminus \gamma\}$ is an
      infinite (1-dimensional) $\Delta$-system with root $v_{d^+, 0}$. Therefore, there exists an $\alpha \in H_0 \setminus \gamma$ such that $v_{d_\alpha} \setminus v_{d^+, 0}$
      is disjoint from $u(r)$. By Claim \ref{ds_claim_3} $q_{d_\alpha, 0} \restriction (v_{d^+, 0} \times \omega) =
      q_{d,0} \restriction (v_{d^+, 0} \times \omega)$; in consequence, since $r$ extends $q_{d,0} = q$, it is compatible with $q_{d_\alpha, 0}$, as desired.
    \end{proof}

    It follows immediately that $q$ forces $\dot{A}$ to have cardinality $\lambda_{n-1}$. To see item (ii) above, notice that $q_{d_\alpha, 0} \Vdash ``
    \dot{f}_\alpha \in \dot{X}"$ for all $\alpha\in H_0$, since $q_{d_\alpha, 0} \leq p_\alpha$. It follows that $\{\dot{f}_\alpha \mid
    \alpha \in \dot{A}\}$ is forced to be a subset of $\dot{X}$.

    We now show that $q$ forces that $\dot{\Phi} \restriction \dot{A}$ is trivial;
    this verification will be considerably more involved. Let $G$ be $\bb{P}$-generic
    over $V$ with $q \in G$. We will begin by working in $V[G]$. To denote the interpretations of $\bb{P}$-names in
    $V[G]$ we simply remove their dots, e.g., the interpretation of $\dot{A}$ is
    $A$, the interpretation of $\dot{\Phi}$ is $\Phi$, etc. Much as above, for
    $a \in [\chi]^{<\omega}$, let $I(a)$ denote $I(\wedge\{f_\alpha \mid \alpha \in a\})$.

    For each $a \in [A]^{\leq n}$ with $|a| > 1$ let $d_a = a \cup
    \{\delta_i \mid |a| - 1 \leq i < n\}$. Notice that $d_a \in [H]^{n+1}$
    (we emphasize that $\delta_n \notin d_a$ in this case). Since $p_\alpha \in G$
    for all $\alpha \in d_a$, there is some unique $\ell_a < \omega$
    for which $q_{d_a, \ell_a} \in G$.

    For each $a \in [A]^{\leq n}$ with $|a| = 1$, i.e., if $a=\{\alpha\}$ for some $\alpha\in A$, let $d_a=d_\alpha$. Again we then have $d_a \in [H]^{n+1}$; since
    $q_{d_\alpha, 0} \in G$ for all $\alpha\in A$, we let $\ell_a = 0$
    for any such singleton $a$.

    For each nonempty $a \in [A]^{\leq n+1}$ let $j_a$ equal
    \[
    \max\{j < \omega \mid \text{there exists a nonempty } a' \in [a]^{\leq n} \text{ such that }
      (\gamma, j) \in \dom(q_{d_{a'}, \ell_{a'}})\text{ for some }\gamma < \chi\}.
    \]
    In particular, there is a single $j^* < \omega$ such that
    \[
      j_{\{\alpha\}} = j^* = \max\{j < \omega \mid \exists \eta < \omega ~ (\eta,j) \in \dom(\bar{q}_0)\}
    \]
    for every $\alpha \in A$.
    For ease of notation, let $I(a)_{>j^*}$ denote
    $\{(j,k) \in I(a) \mid j > j^*\}$.
    Notice that the $j_a$ terms are monotonic: if $a' \subseteq a$, then $j_{a'}
    \leq j_a$.

    We now recursively define ordinals
    \[
      \langle \varepsilon^{j,k}_a \mid a \in [A]^{\leq n+1} \text{ is nonempty and }
      (j,k) \in I(a)_{>j^*} \rangle.
    \]
These ordinals will play, for each $(j,k)\in  I(a)_{>j^*}$, the role of the sequence $\vec{\varepsilon}$ described in Section \ref{defining_trivializations_section}. They will satisfy the following properties:
    \begin{enumerate}
      \item $\varepsilon^{j,k}_a \in H_{|a|-1}$ for all nonempty $a \in [A]^{\leq n+1}$ and all $(j,k) \in I(a)_{>j^*}$;
      \item $\varepsilon^{j,k}_{\{\alpha\}} = \alpha$ for all $\alpha \in A$ and all $(j,k) \in I(f_\alpha)_{>j^*}$;
      \item $(j,k) \in I(f_{\varepsilon^{j,k}_a})$ for all nonempty $a \in [A]^{\leq n+1}$ and all $(j,k) \in I(a)_{>j^*}$.
    \end{enumerate}

    The ordinals $\varepsilon^{j,k}_a$ will satisfy an additional property as well; to state it, we will need some further notation. Suppose that
    $b \in [A]^{\leq n+1}$ is a nonempty set and $\vec{a} = \langle a_i \mid
    1 \leq i \leq m \rangle$ is a subset-final segment of $b$. Suppose also
    that $(j,k) \in I(b)_{>j^*}$. Then, recalling the notation $d^{\vec{\varepsilon}}_{\vec{a}}$ (or $d_{\vec{a}}$, simply) from Section \ref{defining_trivializations_section}, we define
    $d^{j,k}_{\vec{a}} \in [H]^{n+1}$ as follows, splitting into cases depending on whether or
    not $|a_1| > 1$.
     First, if $|a_1| = 1$ (and hence if $|b| = m$), then let
    \[
      d^{j,k}_{\vec{a}} = \{\varepsilon^{j,k}_{a_i} \mid 1 \leq i \leq m\} \cup
      \{\delta_i \mid m \leq i \leq n\}.
    \]
    Notice that in this case $|d^{j,k}_{\vec{a}} \cap H_i| = 1$ for all
    $i \leq n$.

    Next, if $|a_1| > 1$ and $|b| \leq n$, then let
    \[
      d^{j,k}_{\vec{a}} = a_1 \cup \{\varepsilon^{j,k}_{a_i} \mid 1 \leq i \leq m\}
      \cup \{\delta_i \mid |a_1| + m -1 \leq i < n\}.
    \]
    Notice that in this case $|d^{j,k}_{\vec{a}} \cap H_0| = |a_1|$,
    $d^{j,k}_{\vec{a}} \cap H_i = \emptyset$ for $0 < i < |a_1|-1$ and $i = n$,
    and $|d^{j,k}_{\vec{a}} \cap H_i| = 1$ for $|a_1|-1 \leq i < n$. If $|a_1| > 1$
    and $|b| = n+1$, then leave $d^{j,k}_{\vec{a}}$ undefined. In any
    case, if $d^{j,k}_{\vec{a}}$ is defined, then $d^{j,k}_{\vec{a}} \in [H]^{n+1}$.

    We may now state our final requirement for the ordinals $\varepsilon^{j,k}_a$:
    \begin{enumerate}
      \item[(4)] For every nonempty $b \in [A]^{\leq n+1}$, every subset-final segment
      $\vec{a} = \langle a_i \mid 1 \leq i \leq m \rangle$ of $b$, and every $(j,k) \in I(b)_{>j^*}$,
      \begin{enumerate}
        \item if $|a_1| = 1$, then $q_{d^{j,k}_{\vec{a}}, 0} \in G$;
        \item if $|a_1| > 1$ and $|b| \leq n$ and $j > j_b$, then
        $q_{d^{j,k}_{\vec{a}}, \ell_{a_1}} \in G$.
      \end{enumerate}
    \end{enumerate}
    A main resource for the construction of the ordinals $\varepsilon^{j,k}_a$
    will be Claim \ref{compatible_claim} below. To facilitate its statement, we
    introduce the following terminology and convention:
    if $c \in [H]^{n+1}$, $i \leq n$, and $\alpha \in H$, then we say that
    $\alpha$ is \emph{$i$-possible for $c$} if the following two statements hold:
    \begin{itemize}
      \item if $i > 0$, then $\alpha > c(i-1)$;
      \item if $i<n$, then $\alpha < c(i+1)$.
    \end{itemize}
    Intuitively, $\alpha$ is $i$-possible for $c$ if $c(i)$ can be replaced by
    $\alpha$ without changing the positions of the other elements of $c$ within the set.
    If $\alpha$ is $i$-possible for $c$, then $c[i \mapsto \alpha]$ denotes this replacement, i.e., it denotes the
    set $(c \setminus \{c(i)\}) \cup \{\alpha\}$.

    \begin{claim} \label{compatible_claim}
      Suppose that $c_0, c_1 \in [H]^{n+1}$, $i_0, i_1 \leq n$, and $c_0(i_0) = c_1(i_1)$.
      Suppose also that $\alpha \in H$ is $i_0$-possible for $c_0$ and $i_1$-possible for $c_1$
      and $\ell_0, \ell_1 < \omega$ are such that $q_{c_0, \ell_0}$ and
      $q_{c_1, \ell_1}$ are compatible in $\bb{P}$. Then $q_{c_0[i_0 \mapsto \alpha],
      \ell_0}$ and $q_{c_1[i_1 \mapsto \alpha], \ell_1}$ are also compatible in
      $\bb{P}$.
    \end{claim}

    \begin{proof}
      If $c_0(i_0) = \alpha$, then there is nothing to prove, so assume that
      $c_0(i_0) \neq \alpha$. Let $c = c_0 \cup c_1$. Since $c_0$ and $c_1$
      share at least one element, we know that $|c| \leq 2n+1$. Let $b \in [H]^{2n+1}$
      be a (possibly trivial) end-extension of $c$ such that every element of $b \setminus c$
      is greater than $\alpha$. Let $i^* \leq 2n$ be such that
      $b(i^*) = c_0(i_0) = c_1(i_1)$, and let $b^* = (b \setminus \{c_0(i_0)\}) \cup
      \{\alpha\}$. Since $\alpha$ is $i_0$-possible for $c_0$ and $i_1$-possible
      for $c_1$, and since all elements of $b \setminus c$ are greater than $\alpha$,
      we know that $b$ and $b^*$ are aligned and $\mb{r}(b, b^*) = (2n+1)\setminus\{i^*\} =: \mb{m}$.
      Let $\mb{m}_0, \mb{m}_1 \in [2n+1]^{n+1}$ be such that $c_0 = b[\mb{m}_0]$
      and $c_1 = b[\mb{m}_1]$, and hence such that $c_0[i_0 \mapsto \alpha] = b^*[\mb{m}_0]$
      and $c_1[i_1 \mapsto \alpha] = b^*[\mb{m}_1]$.

      Suppose for the sake of contradiction that $q_{c_0[i_0 \mapsto \alpha],
      \ell_0}$ and $q_{c_1[i_1 \mapsto \alpha], \ell_1}$ are incompatible in
      $\bb{P}$. Then there is a $(\gamma^*, j)$ in the intersection of their domains such that
      \[
        q_{c_0[i_0 \mapsto \alpha], \ell_0}(\gamma^*, j) \neq q_{c_1[i_1 \mapsto \alpha],
        \ell_1}(\gamma^*, j).
      \]
      Suppose that $\eta < \rho$ is such that $\gamma^* = u_{b^*}(\eta)$. Let
      $\gamma = u_b(\eta)$. Since $F(b^*) = F(b)$, and in particular since
      $w^b_{\mb{m}_0, \ell_0} = w^{b^*}_{\mb{m}_0, \ell_0}$,
      $w^b_{\mb{m}_1, \ell_1} = w^{b^*}_{\mb{m}_1, \ell_1}$,
      $\bar{q}_{b[\mb{m}_0], \ell_0} = \bar{q}_{b^*[\mb{m}_0], \ell_0}$,
      and $\bar{q}_{b[\mb{m}_1], \ell_1} = \bar{q}_{b^*[\mb{m}_1], \ell_1}$,
      we know that $(\gamma, j)$ is in the domain of both $q_{c_0, \ell_0}$ and
      $q_{c_1, \ell_1}$ and also that $q_{c_0, \ell_0}(\gamma, j) =
      q_{c_0[i_0 \mapsto \alpha], \ell_0}(\gamma^*, j)$ and $q_{c_1, \ell_1}(\gamma, j)
      = q_{c_1[i_1 \mapsto \alpha], \ell_1}(\gamma^*, j)$. It follows
      that $q_{c_0, \ell_0}(\gamma, j) \neq q_{c_1, \ell_1}(\gamma, j)$, but this contradicts
      our assumption that $q_{c_0, \ell_0}$ and $q_{c_1, \ell_1}$ are compatible
      in $\bb{P}$.
    \end{proof}

    We turn now more directly to the construction of the family of ordinals
    \[
      \langle \varepsilon^{j,k}_a \mid a \in [A]^{\leq n+1} \text{ is nonempty and }
      (j,k) \in I(a)_{>j^*} \rangle
    \]
    satisfying the requirements (1)--(4) listed above.
    The construction is by recursion on $|a|$. If $\alpha \in A$ and $(j,k) \in
    I(f_\alpha)_{>j^*}$, then condition (2) dictates that $\varepsilon^{j,k}_{\{\alpha\}}
    = \alpha$. Conditions (1)--(3) are then trivially satisfied. To see condition (4), observe that the only subset-final segment of $\{\alpha\}$ for any $\alpha \in A$ is
    $\vec{a}=\langle \{\alpha\} \rangle$; by the definition of $A$, we then have
    $q_{d^{j,k}_{\vec{a}}, 0} = q_{d_\alpha, 0} \in G$, just as required by condition (4a). This concludes the cases in which $|a|=1$.

    Next suppose that $b \in [A]^{\leq n}$ and $|b| \geq 2$; fix $(j,k) \in I(b)_{>j^*}$ and suppose also that we have defined $\varepsilon^{j,k}_a$ for all nonempty $a \subsetneq b$.
    We will define an $\varepsilon^{j,k}_b \in H_{|b|-1}$.
    Suppose that $\vec{a} = \langle a_i \mid 1 \leq i \leq m \rangle$ is a
    subset-final segment of $b$. If $m = 1$, then define $d_{\vec{a}}^-$ to be $d_b$ (see again
    the third paragraph after Claim \ref{Aunboundedclaim} for the definition of $d_b$). If $m > 1$, then notice that $\vec{a}^- := \langle a_i \mid 1 \leq i
    <m \rangle$ is a subset-final segment of $a_{m-1}$, and define
    $d_{\vec{a}}^-$ to be $d^{j,k}_{\vec{a}^-}$. Observe that $q_{d^-_{\vec{a}}, \ell_{a_1}} \!\in G$ in either case: in the $m=1$ case, this follows from
    the definition of $\ell_{a_1} = \ell_b$. In the $m>1$ case, this follows from
    our inductive hypothesis that statement (4) holds when applied to $a_{m-1}$.
    If $\vec{a}$ is a long string (i.e., if $|a_1| = 1$), then let $i_{\vec{a}} = |b-1|$,
    and if $\vec{a}$ is a short string (i.e., if $|a_1| > 1$), then let
    $i_{\vec{a}} = |b|$. Notice that, in any case, we have $d^-_{\vec{a}}(i_{\vec{a}})
    = \delta_{|b| - 1}$ and, once we have defined $\varepsilon^{j,k}_b$, we will have
    $d^{j,k}_{\vec{a}} = d^-_{\vec{a}}[i_{\vec{a}} \mapsto \varepsilon^{j,k}_b]$.

    Now, to see that we can find an ordinal $\varepsilon^{j,k}_b$ satisfying
    conditions (1)--(4), move back to $V$ and fix an arbitrary $r \in \bb{P}$
    extending
    \[
      q^* := \bigcup \{q_{d^-_{\vec{a}}, \ell_{a_1}} \mid \vec{a} = \langle a_i \mid
      i \leq m \rangle \text{ is a subset-final segment of }b\},
    \]
    which we know to be in $G$. We will find a condition $s \leq r$ and an ordinal
    $\varepsilon$ such that $s$ forces $\varepsilon$ to be a valid choice for
    $\varepsilon^{j,k}_{b}$. By the preceding paragraph and Claim \ref{ds_claim_2},
    we know that for every subset-final
    segment $\vec{a}$ of $b$, the collection $\{v_{d^-_{\vec{a}}[(i_{\vec{a}})
    \mapsto \varepsilon]} \mid \varepsilon \in H_{|b|-1}\}$ is an infinite (1-dimensional)
    $\Delta$-system with root $v_{d^-_{\vec{a}} \setminus \{\delta_{|b|-1}\}, i_{\vec{a}}}$.
    Therefore we can fix an $\varepsilon \in H_{|b| - 1}$ such that
    $\varepsilon \notin u(r)$ and such that $v_{d^-_{\vec{a}}[i_{\vec{a}}
    \mapsto \varepsilon]} \setminus v_{d^-_{\vec{a}} \setminus \{\delta_{|b|-1}\}, i_{\vec{a}}}$
    is disjoint from $u(r)$ for every subset-final segment
    $\vec{a}$ of $b$. For each such $\vec{a}$, Claim \ref{ds_claim_3} implies that
    \[
    q_{d^-_{\vec{a}}[i_{\vec{a}}\mapsto \varepsilon], \ell_{a_1}} \restriction
    (v_{d^-_{\vec{a}} \setminus \{\delta_{|b|-1}\}, i_{\vec{a}}} \times \omega) =
    q_{d^-_{\vec{a}},\ell_{a_1}} \restriction (v_{d^-_{\vec{a}} \setminus \{\delta_{|b|-1}\},
    i_{\vec{a}}} \times \omega),
    \]
    and we know that $r$ extends $q_{d^-_{\vec{a}},\ell_{a_1}}$. Therefore,
    $q_{d^-_{\vec{a}}[i_{\vec{a}}\mapsto \varepsilon], \ell_{a_1}}$
    is compatible with $r$. Moreover, for all pairs $\vec{a}$ and $\vec{a}^*$
    of subset-final segments of $b$, since $q_{d^-_{\vec{a}}, \ell_{a_1}}$ and $q_{d^-_{\vec{a}^*}, \ell_{a^*_1}}$
    are both in $G$ and are therefore compatible, Claim \ref{compatible_claim}
    implies that $q_{d^-_{\vec{a}}[i_{\vec{a}}\mapsto \varepsilon], \ell_{a_1}}$
    and $q_{d^-_{\vec{a}^*}[i_{\vec{a}^*}\mapsto \varepsilon], \ell_{a^*_1}}$ are
    compatible.

    We now split into two cases. Suppose first that $j \leq j_b$, so that we need to
    satisfy requirement (4a) but not (4b). Let
    \[
      s_0 = r \cup \bigcup \{q_{d^-_{\vec{a}}[i_{\vec{a}}\mapsto \varepsilon], \ell_{a_1}}
      \mid \vec{a} \text{ is a subset-final segment of } b \text{ and }
      |a_1| = 1\}.
    \]
    By the previous paragraph, $s_0$ is a condition in $\bb{P}$. Notice also that $\ell_{a_1} = 0$ for all such $\vec{a}$ in the above union, by our inductive condition (4a). By the definition of $j^*$, the fact that $j > j^*$, and the fact that
    $\varepsilon \notin u(r)$, we know that $(\varepsilon, j) \notin \dom(s_0)$.
    Therefore we can extend $s_0$ to a condition $s$ such that
    $(\varepsilon, j) \in \dom(s)$ and $s(\varepsilon, j) \geq k$, i.e.,
    $s \Vdash ``(j,k) \in I(\dot{f}_\varepsilon)"$. This $s$ in fact forces that letting $\varepsilon^{j,k}_a = \varepsilon$ satisfies requirements (1)--(4), as the reader may easily verify.

    If, on the other hand, $j > j_b$, then we need to satisfy both the conditions (4a) and (4b). Let
    \[
      s_0 = r \cup \bigcup \{q_{d^-_{\vec{a}}[i_{\vec{a}}\mapsto \varepsilon], \ell_{a_1}}
      \mid \vec{a} \text{ is a subset-final segment of } b\}.
    \]
    As in the previous case, $s_0$ is a condition in $\bb{P}$. By the definition of
    $j_b$, the fact that $j > j_b$, and the fact that $\varepsilon \notin u(r)$,
    we know that $(\varepsilon, j) \notin \dom(s_0)$. Just as in the previous case,
    we can extend $s_0$ to a condition $s$ such that
    $(\varepsilon, j) \in \dom(s)$ and $s(\varepsilon, j) \geq k$. Also as in the previous case, this $s$ forces
    that letting $\varepsilon^{j,k}_a = \varepsilon$ satisfies requirements (1)--(4), as desired.

    By genericity, our analysis in $V$ shows that we may choose in $V[G]$ an $\varepsilon^{j,k}_b$
    satisfying requirements (1)--(4), and thereby continue with our construction.

    Finally, suppose that $b \in [A]^{n+1}$ and fix $(j,k) \in I(b)_{>j^*}$, and suppose
    that we have defined $\varepsilon^{j,k}_a$ for all nonempty $a \subsetneq b$.
    We will define an $\varepsilon^{j,k}_b \in H_n$; this will be similar to the previous
    case, but we no longer need to satisfy requirement (4b) and can therefore
    focus exclusively on long strings. Suppose that $\vec{a} = \langle a_i \mid
    1 \leq i \leq n \rangle$ is a long string for $b$. Set $\vec{a}^- :=
    \langle a_i \mid 1 \leq i < n \rangle$ and $d_{\vec{a}}^- := d^{j,k}_{\vec{a}^-}$,
    and note that $\vec{a}^-$ is a long string for $a_{n-1}$. As in the previous case,
    we have $q_{d_{\vec{a}}^-, 0} \in G$, $d^-_{\vec{a}}(n) = \delta_n$ and,
    once we have defined $\varepsilon^{j,k}_b$, we will have
    $d^{j,k}_{\vec{a}} = d^-_{\vec{a}}[n \mapsto \varepsilon^{j,k}_b]$.

    To see that we can find an ordinal $\varepsilon^{j,k}_b$ satisfying requirements
    (1)--(4), move back to $V$ and fix an arbitrary $r \in \bb{P}$ extending
    \[
    q^* := \bigcup \{q_{d^-_{\vec{a}}, 0} \mid \vec{a} \text{ is a long string for }b\},
    \]
    which we know to be in $G$. We will find $s \leq r$ and an ordinal $\varepsilon$
    such that $s$ forces $\varepsilon$ to be a valid choice for $\varepsilon^{j,k}_b$.
    By the preceding paragraph and Claim \ref{ds_claim_2}, we know that, for every
    long string $\vec{a}$ for $b$, the collection $\{v_{d^-_{\vec{a}}[n \mapsto
    \varepsilon]} \mid \varepsilon \in H_n\}$ is an infinite $\Delta$-system with
    root $v_{d^-_{\vec{a}} \setminus \{\delta_n\}, n}$. Therefore, we can fix an
    $\varepsilon \in H_n$ such that $\varepsilon \notin u(r)$ and such that
    $v_{d^-_{\vec{a}}[n \mapsto \varepsilon]} \setminus v_{d^-_{\vec{a}} \setminus
    \{\delta_n\}, n}$ is disjoint from $u(r)$ for every long string $\vec{a}$ for
    $b$. For each such $\vec{a}$, Claim \ref{ds_claim_3} implies that
    \[
    q_{d^-_{\vec{a}}[n\mapsto \varepsilon], 0} \restriction
    (v_{d^-_{\vec{a}} \setminus \{\delta_n\}, n} \times \omega) =
    q_{d^-_{\vec{a}},0} \restriction (v_{d^-_{\vec{a}} \setminus \{\delta_n\}, n} \times \omega),
    \]
    and we know that $r$ extends $q_{d^-_{\vec{a}},0}$. Therefore,
    $q_{d^-_{\vec{a}}[n\mapsto \varepsilon], 0}$ is compatible with $r$. Moreover,
    for all pairs $\vec{a}$ and $\vec{a}^*$ of long strings for $b$, since
    $q_{d^-_{\vec{a}}, 0}$ and $q_{d^-_{\vec{a}^*}, 0}$ are both in $G$ and are therefore
    compatible, Claim \ref{compatible_claim} implies that $q_{d^-_{\vec{a}}[n\mapsto
    \varepsilon], 0}$ and $q_{d^-_{\vec{a}^*}[n\mapsto \varepsilon], 0}$ are compatible.

    Now let
    \[
      s_0 = r \cup \bigcup \{q_{d^-_{\vec{a}}[n\mapsto \varepsilon], 0}
      \mid \vec{a} \text{ is a long string for } b\}.
    \]
    By the previous paragraph, $s_0$ is a condition in $\bb{P}$. By the definition of
    $j^*$, the fact that $j > j^*$, and the fact that
    $\varepsilon \notin u(r)$, we know that $(\varepsilon, j) \notin \dom(s_0)$.
    Therefore we can extend $s_0$ to a condition $s$ such that
    $(\varepsilon, j) \in \dom(s)$ and $s(\varepsilon, j) \geq k$, i.e.,
    $s \Vdash ``(j,k) \in I(\dot{f}_\varepsilon)"$. This $s$ in fact forces that letting
    $\varepsilon^{j,k}_a = \varepsilon$ satisfies requirements (1)--(4), as the
    reader may easily verify.

    Suppose now that the construction of the ordinals
     \[
      \langle \varepsilon^{j,k}_a \mid a \in [A]^{\leq n+1} \text{ is nonempty and }
      (j,k) \in I(a)_{>j^*} \rangle
    \] is completed. For all $a \in [A]^n$ and all $(j,k) \in I(a)_{>j^*}$, let $\mathcal{A}^{j,k}_n(a)$
    be defined as in Section \ref{defining_trivializations_section}, using the
    $n$-coherent family $\Phi$ and the ordinals $\langle \varepsilon^{j,k}_{a'}
    \mid a' \in [A]^{\leq n} \rangle$. Similarly define $\mathcal{C}^{j,k}_n(b)$
    for $b \in [A]^{n+1}$ and $(j,k) \in I(b)_{>j^*}$. For $a \in [A]^n$, define a function
    $\psi_a:I(a) \rightarrow \bb{Z}$ as follows. If $(j,k) \in I(a)$ and $j \leq j^*$,
    then let $\psi_a(j,k) = 0$. If $(j,k) \in I(a)_{>j^*}$, then let
    $\psi_a(j,k) = \mathcal{A}^{j,k}_n(a)(j,k)$. Define $\psi_{\vec{a}}$ for
    non-increasing $\vec{\alpha} \in A^n$ in the unique way that renders
    $\Psi = \langle \psi_{\vec{\alpha}} \mid \vec{\alpha} \in A^n \rangle$ an alternating
    family. We claim that $\Psi$ together with the natural number $j^*$ witnesses the triviality
    of $\Phi \restriction A$ in the sense of Fact \ref{trivial_equivalence_fact}.

    We first show that each $\psi_a$ is finitely supported. To see this, fix an arbitrary
    $a \in [A]^n$. For each nonempty $a' \subseteq a$, we have a finite partial
    function $\mathsf{e}_{\ell_{a'}}$ such that, if $\vec{a}$ is a subset-final
    segment of $a$ with $a_1 = a'$, then, for all $(j,k) \in I(a)_{>j^*}$,
    if $q_{d^{j,k}_{\vec{a}}, \ell_{a'}} \in G$, then
    $\mathsf{e}(d^{j,k}_{\vec{a}}) = \mathsf{e}_{\ell_{a'}}$. Fix a natural
    number $j^*_a \geq j_a$ such that, for all nonempty $a' \subseteq a$, we have
    $\dom(\mathsf{e}_{\ell_{a'}}) \subseteq (j^*_a \times \omega)$.

    We claim that $\psi_a(j,k) = 0$ for all $(j,k) \in I(a)_{>j^*_a}$. To see this, fix such
    a pair $(j,k)$. By the definition of $\psi_a(j,k)$, we know that
    $\psi_a(j,k) = \mathcal{A}^{j,k}_n(a)(j,k)$. By Lemma \ref{a_c_form_lemma},
    we know that $\mathcal{A}^{j,k}_n(a)$ is of the form
    \[
      \sum_{i < \ell} c_i e(d^{j,k}_{\vec{a}_i}),
    \]
    where $\ell < \omega$ and each $c_i$ is an integer and each
    $\vec{a}_i$ is a subset-final segment of $a$ with $|a_i(1)| > 1$ (recall that $\vec{a}_i(1)$ denotes the first element of $\vec{a}_i$).
    Moreover, since $j > j^*_a \geq j_a$, we know by condition (4b) that $q_{d^{j,k}_{\vec{a}_i}, \ell_{\vec{a}_i(1)}} \in G$, and hence that
    $\mathsf{e}(d^{j,k}_{\vec{a}_i}) = \mathsf{e}_{\ell_{\vec{a}_i(1)}}$, for all $i < \ell$.
    In particular, $\dom(\mathsf{e}(d^{j,k}_{\vec{a}_i})) \subseteq (j^*_a \times \omega)$,
    so $e(d^{j,k}_{\vec{a}_i})(j,k) = 0$. It follows that $\psi_a(j,k) = 0$.
    In consequence, the support of $\psi_a$ is a subset of $I(a) \cap ((j^*_a+1) \times \omega)$,
    which is a finite set.

    It now only remains to be shown that for all $\vec{\beta} \in A^{n+1}$ and all
    $(j,k) \in I(\vec{\beta})_{>j^*}$,
    \[
      e(\vec{\beta})(j,k) = \sum_{i=0}^n (-1)^i \psi_{\vec{\beta}^i}(j,k).
    \]
    Since $\Phi$ and $\Psi$ are both alternating, it suffices to prove this
    for $b \in [A]^{n+1}$. Fix such a $b$ and a coordinate-pair $(j,k) \in I(b)_{>j^*}$.
    Notice that for every long string $\vec{a}$ for $b$, since $\varepsilon^{j,k}_{\vec{a}}$
    satisfies requirement (4a), we have $q_{d^{j,k}_{\vec{a}, 0}} \in G$
    and hence $\mathsf{e}(d^{j,k}_{\vec{a}}) = \mathsf{e}_0$. Thus we have
    $e(d^{j,k}_{\vec{a}})(j,k) = e_0(j,k)$, where $e_0 : \omega \times \omega
    \rightarrow \bb{Z}$ is the function whose restriction to its support is equal
    to $\mathsf{e}_0$. Moreover, by the construction of $\varepsilon^{j,k}_{a}$
    for nonempty $a \subseteq b$ and the assumption that $(j,k) \in I(b)_{>j^*}$,
    we know that $(j,k) \in I(f_{\varepsilon^{j,k}_a})$ for all nonempty $a \subseteq b$.
    Therefore the hypotheses of Fact \ref{cancellation_fact} hold, and consequently $\mathcal{C}_n^{j,k}(b)(j,k) = 0$. By the definition of
    $\mathcal{C}_n^{j,k}(b)$, we then have
    \[
      0 = \mathcal{C}_n^{j,k}(b)(j,k) = e(b)(j,k) - \sum_{i=0}^n(-1)^i \mathcal{A}^{j,k}_n
      (b^i)(j,k) = e(b)(j,k) - \sum_{i=0}^n (-1)^i \psi_{b^i}(j,k),
    \]
    implying that
    \[
      e(b)(j,k) = \sum_{i=0}^n (-1)^i \psi_{b^i}(j,k),
    \]
    as desired.

    It follows that, in $V[G]$, the restricted family $\Phi \restriction A$ is trivial.
    By our inductive hypotheses, together with the fact that $|A| = \lambda_{n-1}$,  Lemma \ref{summarypropagatinglemma} applies, and we may conclude that $\Phi$ is trivial; this concludes the proof.
\end{proof}
Clearly the theorem stated in our introduction is a special case of Theorem \ref{maintheorem}; observe also that assuming, for example, the generalized continuum hypothesis in our ground model $V$ yields the corollary recorded there as well.
\section{Conclusion}\label{conclusion}
As noted in our introduction, this work fully answers the first question, and partially or potentially addresses the second question, appearing in \cite{SVHDL}. We restate the latter:
\begin{question}\label{Q1}
What is the minimum value of the continuum compatible with the statement ``$\mathrm{lim}^n\,\mathbf{A}=0$ for all $n>0$''?
\end{question}
By our Main Corollary, this question is tantamount to the following:
\begin{question}\label{Q2} Does $2^{\aleph_0}<\aleph_\omega$ imply that $\mathrm{lim}^k\,\mathbf{A}\neq 0$ for some $k>0$?
\end{question}
Answering this question will entail answering the following (a revision, in light of present knowledge, of one appearing in \cite{moorepfa}):
\begin{question} Does $2^{\aleph_0}\leq\aleph_2$ imply that either $\mathrm{lim}^1\,\mathbf{A}\neq 0$ or $\mathrm{lim}^2\,\mathbf{A}\neq 0$?
\end{question}
Of interest in its own right, but all the more so in light of Question \ref{Q2}, is:
\begin{question} What is the behavior of the groups $\mathrm{lim}^n\,\mathbf{A}$ in the standard forcing extensions in which $2^{\aleph_0}=\aleph_2$? By \cite{B1}, of particular interest among them will be those models in which $\mathfrak{b}<\mathfrak{d}$; prominent among these is the Miller model.
\end{question}

The fundamental reason that $2^{\aleph_0}=\aleph_n$ implies $\mathrm{lim}^n\,\mathbf{A}\neq 0$ when $n=1$ is that the answer to the following question is yes when $n=1$ as well.
\begin{question}\label{Q5} Is it a \textsf{ZFC} theorem that any $F\subseteq{^\omega}\omega$ of $<^{*}$-ordertype $\omega_n$ indexes a nontrivial $n$-coherent family?
\end{question}

Recently, Veli\v{c}kovi\'{c} and Vignati \cite{VV} have obtained a positive answer
to Question \ref{Q5} in the presence of additional cardinal arithmetic assumptions.
In particular, they prove that if $2^{\omega_k} < 2^{\omega_{k+1}}$ for all $1 \leq k < n$,
then every $F \subseteq {^\omega}\omega$ of $<^*$-ordertype $\omega_n$ indexes a nontrivial
$n$-coherent family.

A main way of seeing that the answer to Question \ref{Q5} is yes when $n=1$ applies walks techniques to transfer large portions of a nontrivial coherent family on $\omega_1$ to any $F$ as above \cite[pp. 96-98]{bekkali}. Question \ref{Q5} is more generally in large part a question about the combinatorics of the ordinals $\omega_n$ $(n\in\omega)$. Here our researches link up with those of \cite{blh} and \cite{tfoa} in ways we may take the occasion to clarify. A central focus of both those works is nontrivial $n$-coherent families of functions indexed by ordinals $\xi$; much as in the present work, such functions represent nonzero elements of $\mathrm{lim}^n$ of an inverse system $\mathbf{C}(\xi,\mathbb{Z})$, which is defined as follows: for any ordinal $\xi$ and abelian group $A$ let $\mathbf{C}(\xi,A)$ denote the inverse system $(\oplus_{\alpha}A, p_{\alpha\beta},\xi)$ in which the maps $p_{\alpha\beta}:\oplus_\beta A\to \oplus_\alpha A$ are projections for all $\alpha\leq\beta<\xi$. Highly relevant for Question \ref{Q5} are the following facts:
\begin{itemize}
\item $\mathrm{lim}^m\,\mathbf{C}(\omega_n,\mathbb{Z})\neq 0$ for all $n\geq m\geq 0$ in G\"{o}del's constructible universe $L$, as shown in \cite{blh}.
\item There exists (in \textsf{ZFC}) an abelian group $A$ such that $\mathrm{lim}^n\,\mathbf{C}(\omega_n,A)\neq 0$ for all $n\geq 0$, as shown in \cite{tfoa}.
\end{itemize}
Against this background, one of the most central of questions is surely the following:
\begin{question}\label{Q6} Is it a \textsf{ZFC} theorem that $\mathrm{lim}^n\,\mathbf{C}(\omega_n,\mathbb{Z})\neq 0$ for all $n\geq 0$? Put differently, do there exist height-$\omega_n$ nontrivial $n$-coherent families of functions mapping to $\mathbb{Z}$ for all $n>0$ in any model of the \textsf{ZFC} axioms?
\end{question}
Broadly speaking, the argument of \cite{tfoa} is that the fundamental content of a main result from \cite{rings} is the existence of higher-dimensional variants of the walks apparatus first appearing in \cite{todpp}. It seems likely that the answer to Question \ref{Q6} will depend on a better understanding of these higher-dimensional walks, particularly if that answer is yes.
\begin{question}\label{Q7} How much of the classical machinery of walks extends to the $n$-dimensional walks on $\omega_n$ of \cite{tfoa}?
\end{question}
Question \ref{Q5} may be viewed as a special case of Question \ref{Q7}. The prominence of classical coherence phenomena in infinitary combinatorics, as well as the growing prominence of their higher-dimensional variants, is partly explained in \cite{blh} by their connections both to the \v{C}ech cohomology groups of the ordinals and to the broader set-theoretic theme of incompactness. The project of understanding higher-dimensional coherence will in part entail understanding its relation to central incompactness principles like $\square(\kappa)$.
\begin{question} What are the behaviors of $n$-dimensional walks on cardinals $\kappa>\omega_n$, particularly under assumptions like $\square(\kappa)$?
\end{question}
Complementary to the \textsf{ZFC} focus of Questions 5--7 above, in other words, are consistency questions. As the possible behaviors of $\mathrm{lim}^n\,\mathbf{C}(\omega_1,A)$ and $\mathrm{lim}^n\,\mathbf{C}(\omega_2,A)$ are either understood or subsumed by previous questions, the following is among the most immediate:
\begin{question}
Is it consistent with the \textsf{ZFC} axioms that $\mathrm{lim}^2\,\mathbf{C}(\omega_3,A)=0$ for all abelian groups $A$?
\end{question}
Most of the above may be framed as questions about the possible ``spectra'' of nontrivial multidimensional coherence phenomena, or equivalently, of nonvanishing $\mathrm{lim}^n$, either of $\mathbf{A}$ or of $\mathbf{C}(-,-)$. Bound up with these questions seems to be that of the relation of these inverse systems' higher limits to each other. Several other families of inverse systems' higher limits seem to be implicated in these behaviors as well; among the more obvious generalizations of the system $\mathbf{A}$, for example, are those which replace its index-set ${^\omega}\omega$ with ${^\kappa}\lambda$ for arbitrary cardinals $\kappa$ and $\lambda$. As it happens, the vanishing of these systems' higher limits carries implications within the framework of Scholze's \emph{condensed mathematics} \cite{condensed, email}. If $\kappa$ is infinite and $\lambda$ is uncountable, then $\mathrm{lim}^1$ of the associated system is nonzero. The systems in which $\lambda=\omega$, on the other hand, are denoted $\mathbf{A}_\kappa$ in \cite{B1}; there it is shown that $\mathrm{lim}^1\,\mathbf{A}=0$ if and only if $\mathrm{lim}^1\,\mathbf{A}_\kappa=0$ for all $\kappa\geq\omega$. Whether this holds for higher $\mathrm{lim}^n$ is an interesting question, as is the following:
\begin{question}
Let $\kappa$ be an uncountable cardinal. Is it consistent that $\mathrm{lim}^n\,\mathbf{A}_\kappa=0$ for all $n>0$?
\end{question}
A second generalization of the system $\mathbf{A}$ retains the order ${^\omega}\omega$, but varies the groups which it indexes, as well as the homomorphisms connecting them. The work \cite{bbm} isolates a class of such systems significant in strong homology computations; it then shows that arguments applied to $\mathbf{A}$ in \cite{SVHDL} in fact apply to this broader class of systems. This carries the consequence that it is consistent with the \textsf{ZFC} axioms that strong homology is additive on the category of locally compact separable metric spaces; notably, however, these arguments require the existence of a weakly compact cardinal. Somewhat surprisingly, and in contrast to \cite{SVHDL} and \cite{bbm}, there is no straightforward adaptation of the present work's argument to this wider class, for the simple reason that the equivalence of type I and type II triviality so essential to this paper's argument no longer holds in that more general setting.
\begin{question} What is the consistency strength of the statement ``strong homology is additive on the category of locally compact separable metric spaces''?
\end{question}
A last context in which these questions are likely interesting is in the presence of determinacy hypotheses. Relatedly, one might ask how ``definable'' a nontrivial $n$-coherent family of functions indexed by ${^\omega}\omega$ (viewed as a set of real numbers) can be. When $n=1$, such a family is necessarily nonanalytic \cite{todcmpct}; the following question was communicated to the first author by Justin Tatch Moore in 2014.
\begin{question}
Fix $n>1$. Can a nontrivial $n$-coherent family of functions indexed by ${^\omega}\omega$ be analytic?
\end{question}

\bibliographystyle{amsplain}
\bibliography{SVHDLw}

\end{document}